\documentclass[11pt]{amsart}

\usepackage{epigamath}

\usepackage[english]{babel}

\numberwithin{equation}{subsection}

\usepackage[all]{xy}
\usepackage{amscd}

\newtheorem{lemma}{Lemma}[subsection]
\newtheorem{thm}[lemma]{Theorem}
\newtheorem{theorem}{Theorem}
\newtheorem{prop}[lemma]{Proposition}

\newtheorem{cor}[lemma]{Corollary}
\newtheorem{corollary}[lemma]{Corollary}
\newtheorem{lemdef}[lemma]{Lemma-Definition}

\theoremstyle{definition}
\newtheorem{defn}[lemma]{Definition}

\newtheorem{definition}[lemma]{Definition}
\newtheorem{Definition}{Definition}
\newtheorem{nota}[lemma]{Notation}

\theoremstyle{remark}

\newtheorem{rk}[lemma]{Remark}
\newtheorem{remark}[lemma]{Remark}
\newtheorem{remarks}[lemma]{Remarks}
\newtheorem{ex}[lemma]{Example}

\newcommand{\A}{\mathbf{A}}

\newcommand{\N}{\mathbb{N}}

\newcommand{\Z}{\mathbb{Z}}
\newcommand{\sA}{\mathcal{A}}
\newcommand{\sB}{\mathcal{B}}
\newcommand{\sC}{\mathcal{C}}
\newcommand{\sD}{\mathcal{D}}
\newcommand{\sE}{\mathcal{E}}

\newcommand{\sT}{\mathcal{T}}

\newcommand{\bH}{\mathbb{H}}

\newcommand{\bZ}{\mathbb{Z}}

\newcommand{\Mod}{\operatorname{Mod}\hbox{--}}

\newcommand{\Cor}{\operatorname{\mathbf{Cor}}}

\newcommand{\ulSigma}{\ul{\Sigma}}

\newcommand{\MEt}{\operatorname{\mathbf{MEt}}}

\newcommand{\OD}{\mathrm{OD}}

\newcommand{\Ext}{\operatorname{Ext}}
\newcommand{\ul}[1]{{\underline{#1}}}

\newcommand{\PST}{{\operatorname{\mathbf{PST}}}}
\newcommand{\PS}{{\operatorname{\mathbf{PS}}}}
\newcommand{\NS}{{\operatorname{\mathbf{NS}}}}
\newcommand{\NST}{\operatorname{\mathbf{NST}}}

\newcommand{\Hom}{\operatorname{Hom}}

\newcommand{\Ker}{\operatorname{Ker}}

\newcommand{\Coker}{\operatorname{Coker}}

\newcommand{\Sm}{\operatorname{\mathbf{Sm}}}
\newcommand{\Sch}{\operatorname{\mathbf{Sch}}}
\newcommand{\Ab}{\operatorname{\mathbf{Ab}}}

\newcommand{\by}{\xrightarrow}
\newcommand{\yb}{\xleftarrow}
\newcommand{\iso}{\by{\sim}}
\newcommand{\osi}{\yb{\sim}}

\newcommand{\tr}{{\operatorname{tr}}}

\newcommand{\fin}{{\operatorname{fin}}}
\renewcommand{\o}{{\operatorname{o}}}

\newcommand{\red}{{\operatorname{red}}}

\newcommand{\Nis}{{\operatorname{Nis}}}
\newcommand{\et}{{\operatorname{\acute{e}t}}}

\newcommand{\inj}{\hookrightarrow}

\newcommand{\id}{{\operatorname{Id}}}

\renewcommand{\lim}{\operatornamewithlimits{\varprojlim}}
\newcommand{\colim}{\operatornamewithlimits{\varinjlim}}

\newcommand{\hocolim}{\operatorname{hocolim}}
\newcommand{\ol}{\overline}

\renewcommand{\phi}{\varphi}
\renewcommand{\epsilon}{\varepsilon}

\newcommand{\MNS}{\operatorname{\mathbf{MNS}}}
\newcommand{\MNST}{\operatorname{\mathbf{MNST}}}

\newcommand{\MCor}{\operatorname{\mathbf{MCor}}}
\newcommand{\MP}{\operatorname{\mathbf{MSm}}}
\newcommand{\MSm}{\operatorname{\mathbf{MSm}}}
\newcommand{\MPS}{\operatorname{\mathbf{MPS}}}
\newcommand{\MPST}{\operatorname{\mathbf{MPST}}}

\newcommand{\Bl}{{\mathbf{Bl}}}

\newcommand{\Sq}{{\operatorname{\mathbf{Sq}}}}

\newcommand{\Mb}{{\overline{M}}}
\newcommand{\Nb}{{\overline{N}}}
\newcommand{\Lb}{{\overline{L}}}
\newcommand{\Zb}{{\overline{Z}}}

\newcommand{\ulMP}{\operatorname{\mathbf{\underline{M}Sm}}}
\newcommand{\ulMSm}{\operatorname{\mathbf{\underline{M}Sm}}}
\newcommand{\ulMNS}{\operatorname{\mathbf{\underline{M}NS}}}
\newcommand{\ulMPS}{\operatorname{\mathbf{\underline{M}PS}}}
\newcommand{\ulMPST}{\operatorname{\mathbf{\underline{M}PST}}}
\newcommand{\ulMNST}{\operatorname{\mathbf{\underline{M}NST}}}
\newcommand{\ulMCor}{\operatorname{\mathbf{\underline{M}Cor}}}
\newcommand{\ulomega}{\underline{\omega}}

\newcommand{\Comp}{\operatorname{\mathbf{Comp}}}
\newcommand{\CompM}{\Comp(M)}

\newcommand{\MV}{\operatorname{MV}}
\newcommand{\ulMV}{\operatorname{\underline{MV}}}
\newcommand{\ulMVfin}{\operatorname{\underline{MV}^{\mathrm{fin}}}}

\def\Zp{\Z^p}

\def\bZ{\mathbb{Z}}
\def\Ztr{\bZ_\tr}

\def\Ninf{N^\infty}
\def\Minf{M^\infty}
\def\Ztrtau#1{\bZ_\tr(#1)^\tau}
\def\Ztrfintau#1{\Ztr^{\fin}(#1)^{\tau}}
\def\ulMCortau{\ulMCor^\tau}

\def\Minf{M^\infty}
\def\Xlam{X_{\lambda}}
\def\Xmu{X_{\mu}}

\newcounter{spec}
\newenvironment{thlist}{\begin{list}{\rm{(\roman{spec})}}%
{\usecounter{spec}\labelwidth=20pt\itemindent=0pt\labelsep=10pt}}%
{\end{list}}%

\def\aNis{a_{\Nis}}
\def\ulaNis{\underline{a}_{\Nis}}
\def\ulasNis{\underline{a}_{s,\Nis}}
\def\ulaNisfin{\underline{a}^{\fin}_{\Nis}}

\def\asNis{a_{s,\Nis}}
\def\ulasNis{\underline{a}_{s,\Nis}}
\def\qaq{\quad\text{ and }\quad}

\def\Comp{\Comp^{\fin}}

\def\MSm{\operatorname{\mathbf{MSm}}}

\def\ulMSm{\operatorname{\mathbf{\ul{M}Sm}}}

\def\ulMPS{\operatorname{\mathbf{\ul{M}PS}}}

\def\ulMNS{\operatorname{\mathbf{\ul{M}NS}}}

\def\ulMNST{\operatorname{\mathbf{\ul{M}NST}}}
\def\MNS{\operatorname{\mathbf{MNS}}}
\def\MNST{\operatorname{\mathbf{MNST}}}

\def\Comp{\operatorname{\mathbf{Comp}}}
\def\uli{\ul{i}}

\EpigaVolumeYear{5}{2021} \EpigaArticleNr{2} \ReceivedOn{December
17, 2019} \InFinalFormOn{November 26, 2020} \AcceptedOn{December 19,
2020}

\title{Motives with modulus, II: \\ Modulus sheaves with transfers for proper modulus pairs}
\titlemark{Motives with modulus, II}

\author{Bruno Kahn}
\address{IMJ-PRG, Case 247, 4 place Jussieu, 75252 Paris Cedex 05, France}
\email{bruno.kahn@imj-prg.fr}
\author{Hiroyasu Miyazaki}
\address{RIKEN iTHEMS, Wako, Saitama 351-0198, Japan}
\email{hiroyasu.miyazaki@riken.jp}
\author{Shuji Saito}
\address{Graduate School of Mathematical Sciences, University of Tokyo, 3-8-1 Komaba, Tokyo 153-8941, Japan}
\email{sshuji@msb.biglobe.ne.jp}
\author{Takao Yamazaki}
\address{Institute of Mathematics, Tohoku University, Aoba, Sendai 980-8578, Japan}
\email{ytakao@math.tohoku.ac.jp}

\authormark{B. Kahn, H. Miyazaki, S. Saito, and T. Yamazaki}

\AbstractInEnglish{We develop a theory of sheaves and cohomology on
the category of proper modulus pairs. This complements \cite{kmsy},
where a theory of sheaves and cohomology on the category of
non-proper modulus pairs has been developed.}

\MSCclass{19E15; 14F42; 19D45; 19F15}

\KeyWords{modulus pair; presheaf with transfers; cd-structure}

\TitleInFrench{Motifs avec modules, II: faisceaux avec transferts pour les couples modulaires propres}

\AbstractInFrench{Nous pr\'esentons une th\'eorie de faisceaux sur la cat\'egorie des couples modulaires propres, et d\'ecrivons leur cohomologie. Ceci compl\`ete \cite{kmsy}, o\`u  a \'et\'e pr\'esent\'ee une th\'eorie de faisceaux sur la cat\'egorie des couples modulaires non propres.}

\acknowledgement{The first author acknowledges the support of Agence
Nationale de la Recherche (ANR) under reference ANR-12-BL01-0005.
The work of the second author is supported by RIKEN iTHEMS, and by
JSPS KAKENHI Grant (19K23413). The third author is supported by JSPS
KAKENHI Grant (15H03606). The fourth author is supported by JSPS
KAKENHI Grant (15K04773).}

\begin{document}


\removeabove{0.5cm} \removebetween{0.5cm} \removebelow{0.6cm}

\maketitle

\begin{prelims}

\DisplayAbstractInEnglish

\bigskip

\DisplayKeyWords

\medskip

\DisplayMSCclass

\bigskip

\languagesection{Fran\c{c}ais}

\bigskip

\DisplayTitleInFrench

\medskip

\DisplayAbstractInFrench

\end{prelims}


\newpage

\setcounter{tocdepth}{1}

\tableofcontents


\section*{Introduction}\addcontentsline{toc}{section}{Introduction}

This is a sequel to \cite{kmsy}, where a theory of sheaves and
cohomology on the category $\ulMCor$ of non-proper modulus pairs has
been developed. This paper complements it by using work from
\cite{KaMi} and \cite{nistopmod} to develop a theory of sheaves and
cohomology on the category $\MCor$ of proper modulus pairs. This
completes the repairs to the mistake in \cite{motmod}. The basic aim
of both works is to lay a foundation for a theory of \emph{motives
with modulus}, to be completed in \cite{motmod2}, generalizing
Voevodsky's theory of motives in order to capture non
$\A^1$-invariant phenomena.

In \cite{kmsy}, Voevodsky's category $\Cor$ of finite
correspondences on smooth separated schemes of finite type over a
fixed base field $k$, was enlarged to the larger category of
(non-proper) \emph{modulus pairs}, $\ulMCor$. Objects are pairs
$M=(\ol{M}, \Minf)$ consisting of a separated $k$-scheme of finite
type $\ol{M}$ and an effective (possibly empty) Cartier divisor
$\Minf$ on it such that the complement $M^\circ:=\ol{M}\setminus
\Minf$ is in $\Sm$ (we call it the \emph{smooth interior}). The
group $\ulMCor(M,N)$ of morphisms is defined as the subgroup of
$\Cor(M^\circ,N^\circ)$ consisting of finite correspondences between
smooth interiors whose closures in $\ol{M} \times_k \ol{N}$ are
proper\footnote{Here we stress that we do not assume it is finite
over $\ol{M}$.} over $\ol{M}$ and satisfy certain admissibility
conditions with respect to $\Minf$ and $\Ninf$. Let $\MCor\subset
\ulMCor$ be the full subcategory consisting of objects
$(\ol{M},\Minf)$ with $\ol{M}$ proper over $k$.

We then define $\ulMPST$ (resp. $\MPST$) as the category of additive
presheaves of abelian groups on $\ulMCor$ (resp. $\MCor$). We have a
pair of adjunctions
\begin{equation}\label{eq1}
\MPST\begin{smallmatrix} \tau^*\\ \longleftarrow\\ \tau_!\\ \longrightarrow\\
\end{smallmatrix}\ulMPST,
\end{equation}
where $\tau^*$ is induced by the inclusion $\tau:\MCor\to \ulMCor$
and $\tau_!$ is its left Kan extension (see Lemma \ref{eq.tau}).

The main aim of \cite{kmsy} was to develop a \emph{sheaf theory} on
$\ulMCor$ generalizing Voevodsky's theory of sheaves on $\Cor$.

\begin{Definition}\label{defintro;sheavesulMCor}
We define $\ulMNST$ to be the full subcategory of $\ulMPST$ of 
objects $F$ such that $F_M$ is a Nisnevich sheaf on $\Mb$ for every
$M=(\Mb,\Minf)\in \ulMCor$, where $F_M$ is the presheaf on
$\Mb_\Nis$ which associates $F(U,\Minf\times_{\Mb} U)$ to an \'etale
map $U\to \Mb$.
\end{Definition}

\begin{Definition}\label{def;ulSigmafin}
Let $\ul{\Sigma}^\fin$ be the subcategory of $\ulMCor$ which have
the same objects as $\ulMCor$ and such that a morphism $f \in
\ulMCor(M, N)$ belongs to $\ul{\Sigma}^\fin$ if and only if $f^\o
\in \Cor(M^\o, N^\o)$ is the graph of an isomorphism $M^\o \simeq
N^\o$ in $\Sm$ that extends to a proper morphism $\ol{f} : \ol{M}
\to \ol{N}$ of $k$-schemes such that $\Minf=\ol{f}^*\Ninf$.
\end{Definition}

Now the main result of \cite{kmsy} is the following.

\begin{theorem}[\protect{\cite[Theorem~2]{kmsy}}]\label{thmintro;ulMNST}
The following assertions hold.
\begin{itemize}
\item[\rm (1)]
The inclusion $\ulMNST\to \ulMPST$ has an exact left adjoint
$\ulaNis$ such that
\begin{equation*}
(\ulaNis F)(M) =\colim_{N\in \ul{\Sigma}^\fin\downarrow M}
(F_N)_\Nis(N)
\end{equation*}
for every $F\in \ulMPST$ and $M \in \ulMCor$, where $(F_M)_\Nis$ is
the Nisnevich sheafification of the preshseaf $F_M$ on $\Mb_\Nis$.
In particular $\ulMNST$ is a Grothendieck abelian category.
\item[\rm (2)]
For $M\in \ulMCor$, let $\Ztr(M)=\ulMCor(-,M) \in \ulMPST$ be the
associated representable presheaf. Then we have $\Ztr(M)\in \ulMNST$
and there is a canonical isomorphism for any $i\ge 0$ and $F\in
\ulMNST$:
\[\Ext^i_{\ulMNST}(\Z_\tr(M),F)\simeq \colim_{N\in \ul{\Sigma}^{\fin}\downarrow M}
H_\Nis^i(\Nb,F_N).
\]
\end{itemize}
\end{theorem}

The aim of the present paper is to introduce a sheaf theory on
$\MCor$.

\begin{Definition}\label{defintro;sheavesMCor}
We define $\MNST$ to be the full subcategory of $\MPST$ of 
objects $F$ such that $\tau_!F \in \ulMNST$.
\end{Definition}

Note that by definition, $\tau_!:\MPST \to \ulMPST$ induces a
functor
\[\tau_\Nis:\MNST \to \ulMNST.\]
Now the main result of this paper is the following.

\begin{theorem}[see Lemmas \ref{l4.2} and \ref{lem:tau!-exact}, Theorems \ref{thm:a-nis-final}, \ref{thm;tauexact} and \ref{thm:coh-MNST}]\label{thmintro;MNST}\
\begin{itemize}
\item[\rm (1)]
We have $\tau^*(\ulMNST)\subset \MNST$. Letting
\[ \tau^\Nis:\ulMNST \to \MNST\]
be the induced functor, the pair of adjoint functors $(*)$ induces a
pair of adjoint functors
\begin{equation*}\label{adjunction;tau-omegaNis}
\MNST\begin{smallmatrix} \tau^\Nis\\ \longleftarrow\\ \tau_\Nis\\
\longrightarrow
\end{smallmatrix}\ulMNST,\\
\end{equation*}
where $\tau_\Nis$ is exact and fully faithful. The functor
$\tau^\Nis$ is also exact.
\item[\rm (2)]
The inclusion $\MNST\to \MPST$ has an exact left adjoint $\aNis$
such that $\ulaNis\tau_! = \tau_\Nis \aNis$. In particular, $\MNST$
is a Grothendieck abelian category.
\item[\rm (3)]
For $M\in \MCor$, let $\Ztr(M)=\MCor(-,M) \in \MPST$ be the
associated representable presheaf. Then we have $\Ztr(M)\in \MNST$
and there is a canonical isomorphism for any $i\ge 0$ and $F\in
\MNST$:
\[\Ext^i_{\MNST}(\Z_\tr(M),F)\simeq \colim_{N\in \Sigma^{\fin}\downarrow M}
H_\Nis^i(\Nb,(\tau_\Nis F)_N),
\]
where $\Sigma^\fin := \ul{\Sigma}^\fin \cap \MCor$.
\end{itemize}
\end{theorem}

Finally we explain relations between cohomologies for $\MNST$ and
$\NST$. We denote by $\PST$ (resp. $\NST$) Voevodsky's category of
presheaves (resp. Nisnevich sheaves) with transfers. The functor given by
$\omega : \MCor \to \Cor,~ \omega(\ol{M}, M^\infty)=\ol{M} \setminus
|M^\infty|$ induces a pair of adjunctions
\begin{equation*}
\MPST\begin{smallmatrix} \omega^*\\ \longleftarrow\\ \omega_!\\ \longrightarrow\\
\end{smallmatrix}\PST,
\end{equation*}
in a similar way as \eqref{eq1}. (See \S \ref{sect:NST} for
details.)

\begin{theorem}[see Proposition \ref{prop:MNST-NST} and Theorem \ref{prop:MNST-NST-coh}]\label{thm;MNSTNST}
The following assertions hold.
\begin{itemize}
\item[\rm (1)]
We have
\[\omega_!(\MNST)\subset \NST \qaq \omega^*(\NST)\subset \MNST.\]
The functors $\omega_!$ and $\omega^*$ induce a pair of adjoint
functors
\begin{equation*}
\MNST
\begin{smallmatrix} \omega^\Nis\\ \longleftarrow\\ \omega_\Nis\\ \longrightarrow\\
\end{smallmatrix}\NST,
\end{equation*}
such that
\begin{align*}
& \omega_\Nis a_\Nis = a_\Nis^V \omega_!, \quad \omega^\Nis a_\Nis^V
= a_\Nis \omega^*,
\end{align*}
where $a_\Nis^V:\PST\to \NST$ is Voevodsky's sheafification functor.
Moreover, $\omega^\Nis$ and $\omega_\Nis$ are exact and
$\omega^\Nis$ is fully faithful.
\item[\rm (2)]
For any $F\in \MNST$ and $X\in \Sm$ and $i\geq 0$, we have a
canonical isomorphism
\[ H_\Nis^i(X,\omega_\Nis F) \simeq \colim_{M\in \MSm(X)} H_\Nis^i(\Mb,F_M)\]
where $\MSm(X)=\{M\in \MCor\mid M^\o=X\}$ is viewed as a cofiltered
ordered set \cite[Lemma 1.7.4]{kmsy}.
\end{itemize}
\end{theorem}

A key ingredient of the proofs is Theorem \ref{cd-str-MSm}, which is
based on the works \cite{nistopmod} and \cite{KaMi}.

\subsection*{Acknowledgements}
Part of this work was done while the first author was visiting RIKEN
iTHEMS under the invitation of the second author: the first and the
second authors wish to thank both for their hospitality and
excellent working conditions. Part of this work was done while the
third author stayed at the university of Regensburg supported by the
SFB grant ``Higher Invariants". The third author is grateful to the
support and hospitality received there. We also thank the referee
for a thorough reading.

The first author thanks Joseph Ayoub for explaining him an easy but
crucial result on unbounded derived categories (Lemma \ref{layoub}).

\subsection*{Notation and conventions}
In the whole paper we fix a base field $k$. Let $\Sch$ be the
category of separated schemes of finite type over $k$, and let $\Sm$
be its full subcategory of smooth schemes. We write $\Cor$ for
Voevodsky's category of finite correspondences \cite{voetri}.

An additive functor between additive categories is called
\emph{strongly additive} if it commutes with all representable
direct sums. A Grothendieck topology is called \emph{subcanonical}
if every representable presheaf is a sheaf.

Let $\sC$ and $\sD$ be sites, and $u : \sC \to \sD$ a functor. We
say that $u$ is \emph{continuous} (resp. \emph{cocontinuous}) if the
functor $u^*:\hat\sD\to \hat\sC$ (resp. $u_*:\hat\sC\to \hat \sD$)
between categories of presheaves carries sheaves to sheaves. (\emph{cf.}
\cite[expos\'e~III]{SGA4} and \cite[\S A.1]{KaMi}.)

\section{Review of presheaf theory on modulus pairs}\label{s1}

\subsection{Categories of modulus pairs}

A \emph{modulus pair} $M$ consists of $\ol{M} \in \Sch$ and an
effective Cartier divisor $M^\infty \subset \ol{M}$ such that the
open subset $M^\o:=\ol{M} - |M^\infty|$ is smooth over $k$. (The
case $|M^\infty|=\emptyset$ is allowed.) We say that $M$ is
\emph{proper} if $\Mb$ is. Note  that $\ol{M}$ is automatically
reduced, and $M^\o$ is dense in $\ol{M}$ \cite[Remark~1.1.2
(3)]{kmsy}.

Let $M_1, M_2$ be  modulus pairs. Let $Z \in \Cor(M_1^\o, M_2^\o)$
be an elementary (\emph{i.\,e.} integral) finite correspondence in the sense of
Voevodsky \cite{voetri}. We write $\ol{Z}^N$ for the normalization
of the closure $\ol{Z}$ of $Z$ in $\ol{M}_1 \times \ol{M}_2$ and
$p_i : \ol{Z}^N \to \ol{M}_i$ for the canonical morphisms for $i=1,
2$. We say $Z$ is \emph{admissible} (resp. \emph{left-proper}) for
$(M_1, M_2)$ if $p_1^* M_1^\infty \geq p_2^* M_2^\infty$ (resp.
$\ol{Z}$ is proper over $\ol{M}_1$).

By \cite[Proposition~1.2.3 and 1.2.6]{kmsy}, modulus pairs and left proper
admissible correspondences define an additive category that we
denote by $\ulMCor$. We write $\MCor$ for the full subcategory of
$\ulMCor$ whose objects  are proper modulus pairs.

We write $\ulMP$ for the category with same objects as $\ulMCor$, a
morphism of $\ulMP(M_1,M_2)$ being a (scheme-theoretic) $k$-morphism
$f^\o:M_1^\o\to M_2^\o$ whose graph belongs to $\ulMCor(M_1,M_2)$.
We write $\MP$ for the full subcategory of $\ulMP$ whose objects are
proper modulus pairs.

We write $\ulMCor^\fin$ for the subcategory of $\ulMCor$ with the
same objects and the following condition on morphisms: $\alpha\in
\ulMCor(M,N)$ belongs to $\ulMCor^\fin(M,N)$ if and only if, for any
component $Z$ of $\alpha$, the projection $\Zb\to \Mb$ is
\emph{finite}, where $\ol{Z}$ is the closure of $Z$ in $\ol{M}
\times \ol{N}$. We write $\MCor^\fin$ for the full subcategory of
$\ulMCor$ whose objects  are proper modulus pairs.

We write $\ulMP^\fin$ for the subcategory of $\ulMSm$  with the same
objects and such that a morphism $f:M\to N$ belongs to $\ulMP^\fin$
if and only if $f^\o : M^\o  \to N^\o$ extends to a $k$-morphism
$\ol{f} : \ol{M}\to \ol{N}$. We write $\MSm^\fin$ for the full
subcategory of $\ulMSm$ whose objects  are proper modulus pairs. A
morphism $f : M \to N$ in $\ulMSm^\fin$ is \emph{minimal} if we have
$f^\ast N^\infty = M^\infty$.

\begin{remarks}\label{rk-graph-trick}\
\begin{enumerate}
\item
For $M \in \ulMSm^\fin$, set $M^N:=(\ol{M}^N, M^\infty
|_{\ol{M}^N})$ where $p:\ol{M}^N \to \ol{M}$ is the normalization
and $M^\infty |_{\ol{M}^N}$ is the pull-back of $M^\infty$ to
$\ol{M}^N$. Then  $p : M^N \to M$ is an isomorphism in
$\ulMCor^\fin$ and $\ulMSm$ (but not in $\ulMSm^\fin$ in general).
\item
Let $f:M\to N$ be a morphism in $\ulMP^\fin$. The reducedness of
$\ol{M}$, the separatedness of $\ol{N}$ and the denseness of $M^\o $
in $\ol{M}$ imply that this extension $\ol{f}$ is unique. This
yields a forgetful functor $\ulMSm^\fin \to \Sch$, which sends $M$
to $\ol{M}$.
\end{enumerate}
\end{remarks}

We have the following commutative diagram of inclusion functors
\begin{equation}\label{MPcategories}
\vcenter {\xymatrix{ \ulMSm^\fin \ar[r]^{\ul{b}_s}
\ar[d]^{\ul{c}^\fin} & \ulMSm \ar[d] ^{\ul{c}} &
\ar[l]_{\tau_s} \ar[d]^{c} \MSm\\
\ulMCor^\fin \ar[r]^{\ul{b}}  & \ulMCor &\ar[l]_{\tau} \MCor } }
\end{equation}

\subsection{Presheaves}

\begin{definition}\label{d2.7}
By a presheaf we mean here an additive contravariant functor to the
category of abelian groups. (A functor is called additive if it
commutes with finite coproducts.)
\begin{enumerate}
\item
The category of presheaves on $\MP$ (resp. $\ulMP$, $\ulMP^\fin$) is
denoted by $\MPS$ (resp. $\ulMPS$, $\ulMPS^\fin$).
\item
The category of presheaves on $\MCor$ (resp. $\ulMCor$,
$\ulMCor^\fin$) is denoted by $\MPST$ (resp. $\ulMPST$,
$\ulMPST^\fin$).
\item
We write
\begin{align*}
\Z_\tr:& \ulMCor\to\ulMPST, \quad \MCor\to \MPST,
\\
\Z_\tr^\fin:& \ulMCor^\fin\to \ulMPST^\fin,
\end{align*}
for the associated representable presheaf functors.
\item
For $M \in \ulMSm$, we denote by $\Zp (M)$ the presheaf with values
in abelian groups defined by
\[
\ulMSm \ni N \mapsto \Z \ulMSm (N,M),
\]
where for any set $S$ we denote by $\Z S$ the free abelian group on
$S$.
\end{enumerate}
\end{definition}

Diagram \eqref{MPcategories} induces a commutative diagram of
functors on presheaf categories:
\begin{equation}\label{MPScategories}
\vcenter{ \xymatrix{
\ulMPS^\fin & \ar[l]_{\ul{b}_s^*} \ulMPS \ar[r]^{\tau_s^*} & \MPS\\
\ulMPST^\fin  \ar[u]_{\ul{c}^{\fin *}} & \ar[l]_{\ul{b}^*} \ulMPST
\ar[u]_{\ul{c}^*} \ar[r]^{\tau^*} & \MPST \ar[u]_{c^*} } }
\end{equation}

\begin{lemdef}[\protect{\cite[Definition~1.8.1 and Lemma~1.8.2]{kmsy}}]\label{d1.12} For $M=(\Mb,M^\infty)\in \ulMSm$, we denote by $\Comp(M)$ the category whose objects are morphisms $M\by{j_N} N$ in $\ulMSm^\fin$ such that
\begin{thlist}
\item $N\in \MSm$;
\item $\ol{j_N}:\ol{M}\to \ol{N}$ is a dense open immersion;
\item $j_N$ is minimal, i.e., $M^\infty = \ol{j_N}^*N^\infty$;
\item we have $N^\infty=M_N^\infty+C$
for some effective Cartier divisors $M_N^\infty, C$ on $\ol{N}$
satisfying $\ol{N} \setminus |C| = j(\ol{M})$. $($Note that,
therefore, $M^\infty = \ol{j_N}^*M_N^\infty$.$)$
\end{thlist}
For $N_1,N_2\in \Comp(M)$ we define
\[\CompM(N_1,N_2)=\{\gamma\in \MSm(N_1,N_2)\;|\; \gamma\circ j_{N_1}  =j_{N_2}\}.\]
The category $\Comp(M)$ is nonempty, ordered and cofiltered.
\end{lemdef}

\begin{lemma}[\protect{\cite[Proposition~2.4.1 and Lemma~2.4.2]{kmsy}}]\label{eq.tau}\
\begin{itemize}
\item[\rm (1)]
The functor $\tau:\MCor\to \ulMCor$ of  \eqref{MPcategories} yields
a string of $3$ adjoint functors $(\tau_!,\tau^*,\tau_*)$:
\[\MPST\begin{smallmatrix}\tau_!\\\longrightarrow\\\tau^*\\\longleftarrow\\\tau_*\\ \longrightarrow\end{smallmatrix}\ulMPST, \]
where $\tau_!,\tau_*$ are fully faithful, $\tau^*$ is a localisation
and the adjunction map $\id \to \tau^* \tau_!$ is an isomorphism.
The functors $\tau_!$ and $\tau^*$ commute with all colimits and
$\tau_!$ has a pro-left adjoint represented by $\Comp$, hence is
exact.

\item[\rm (2)]
The same statements as (1) hold for the functor $\tau_s:\MSm\to
\ulMSm$.
\item[\rm (3)]
For $G \in \MPST$ and $M \in \ulMP$, we have
\[ \tau_!G(M) \simeq \colim_{N\in \Comp(M)} G(N).  \]
\item[\rm (4)]
For $G \in \MPS$ and $M \in \ulMP$, we have
\[ \tau_{s !}G(M)\simeq \colim_{N\in \Comp(M)} G(N) .
\]
\end{itemize}
\end{lemma}

\begin{lemma}[\protect{\cite[Proposition~2.5.1]{kmsy}}]  \label{eq:bruno-functor}
Let $\ul{b}:\ulMCor^\fin\to \ulMCor$ be the inclusion functor from
\eqref{MPcategories}. Then $\ul{b}$ yields a string of $3$ adjoint
functors $(\ul{b}_!,\ul{b}^*,\ul{b}_*)$:
\[\ulMPST^\fin\begin{smallmatrix}\ul{b}_!\\\longrightarrow\\ \ul{b}^*\\\longleftarrow\\ \ul{b}_*\\ \longrightarrow\end{smallmatrix}\ulMPST, \]
where $\ul{b}_!,\ul{b}_*$ are localisations; $\ul{b}^*$ is exact and
fully faithful; $\ul{b}_!$ has a pro-left adjoint, hence is exact.
For $F\in \ulMPST^\fin$ and $M\in \ulMCor$, we have (see Definition
\ref{def;ulSigmafin})
\begin{equation}\label{eq:b-sh-explicit}
\ul{b}_! F(M) = \colim_{N\in \ul{\Sigma}^{\fin}\downarrow M} F(N).
\end{equation}
The same statements hold for $\ul{b}_s$ from \eqref{MPcategories}.
\end{lemma}

\begin{lemma}[\protect{\cite[Proposition~2.6.1]{kmsy}}]  \label{eq:c-functor}
Let $\ul{c} :\ulMSm \to \ulMCor$ be the functor from
\eqref{MPcategories}. Then $\ul{c}$ yields a string of $3$ adjoint
functors $(\ul{c}_!,\ul{c}^*,\ul{c}_*)$:
\[\ulMPS\begin{smallmatrix}\ul{c}_!\\\longrightarrow\\ \ul{c}^*\\\longleftarrow\\ \ul{c}_*\\ \longrightarrow\end{smallmatrix}\ulMPST,
\]
where $\ul{c}^*$ is exact and faithful (but not full). The same
statements hold for $\ul{c}^\fin$ and $c$ from \eqref{MPcategories}.
We have
\begin{equation}\label{eq:c-and-tau}
c^* \tau^* = \tau_s^* \ul{c}^*, \quad \ul{c}^* \tau_! = \tau_{s !}
c^*,
\end{equation}
\begin{equation}\label{eq:b-and-c}
\ul{c}^{\fin *} \ul{b}^* = \ul{b}_s^* \ul{c}^*, \quad \ul{b}_!
\ul{c}^\fin_! = \ul{c}_! \ul{b}_{s !}, \quad \ul{c}^* \ul{b}_!=
\ul{b}_{s !} \ul{c}^{\fin *}.
\end{equation}
\end{lemma}

For the readers' convenience, we recall the following lemma from
\cite[Lemma A.8.1]{kmsy}:

\begin{lemma}\label{lem:lr-adjoint}
Let $\sC,  \sD$ be abelian categories and let $\sC' \subset \sC,
\sD' \subset \sD$ be full abelian subcategories. Let $c :\sC \to
\sD$ and $c' :\sC' \to \sD'$ be additive functors satisfying $c i_C
= i_D c'$, where $i_C : \sC' \to \sC$ and $i_D : \sD' \to \sD$ are
the inclusion functors.
\begin{itemize}
\item[\rm (1)] If $c$ is faithful, so is $c'$.
\item[\rm (2)] Suppose that $i_D$ is strongly additive or  has a strongly additive left inverse (for example, a left adjoint). If $c$ and $i_C$ are strongly additive, so is $c'$.
\item[\rm (3)] Suppose that $i_C$ has a left adjoint $a_C$.
If $c$ has a left adjoint $d$, then $d'=a_C d i_D$ is a left adjoint
of $c'$. If $d$ and $a_\sC$ are exact, so is $d'$.  Moreover, $a_\sC
d=d'a_\sD$  if $i_\sD$ has a left adjoint $a_\sD$.
\item[\rm (4)]
Suppose that $i_C$ and $i_D$ have left adjoints $a_C$ and $a_D$,
that $a_D$ is exact, and that $a_D c= c' a_C$. If $c$ is exact, then
so is $c'$.
\end{itemize}
\end{lemma}



\section{Review of sheaf theory on non-proper modulus pairs}

In this section we recall some basic definitions and properties on
sheaves on categories of non-proper modulus pairs from \cite[\S4.1
and \S4.2]{kmsy}.

Let $\Sq$ be the product category of $[0]=\{ 0 \to 1 \}$ with
itself, depicted as
\[
\xymatrix{
00 \ar[r] \ar[d] &01 \ar[d]\\
10 \ar[r] & 11. }
\]
For any category $\sC$, denote by $\sC^\Sq$ for the category of
functors from $\Sq$ to $\sC$. A functor $f : \sC \to \sC'$ induces a
functor $f^\Sq : \sC^\Sq \to {\sC'}^\Sq$.

Let us consider $Q \in \ulMSm^\Sq$. We write $Q(ij) = (\ol{Q}(ij), Q^\infty
(ij))$ for all $i,j \in \{0,1\}$. We also write $Q^\o (ij) :=
\ol{Q}(ij) - |Q^\infty (ij)|$.

\subsection{The $\protect\ulMV^\fin$ cd-structure}

\begin{defn}\label{d3.2} \
\begin{enumerate}
\item
A Cartesian square
\begin{equation}\label{eq.cd}
\begin{CD}
W@>v>> V\\
@VqVV @Vp VV\\
U@>u>> X
\end{CD}
\end{equation}
in $\Sch$ is called an \emph{elementary Nisnevich square} if  $p$ is \'etale and $p^{-1}(X \setminus U)_\red\to (X \setminus U)_\red$ is an isomorphism and if $u$ is
an open embedding. In this situation, we
say $U \sqcup V \to X$ is an \emph{elementary Nisnevich cover}.
Recall that  an additive presheaf is a Nisnevich sheaf if and only
if it transforms any elementary Nisnevich square into a cartesian
square \cite[Corollary~2.17]{cdstructures}, \cite[Theorem~2.2]{unstableJPAA}.
\item
A diagram \eqref{eq.cd} in $\ulMP^\fin$ is called an
\emph{$\ulMVfin$-square} if it becomes an elementary Nisnevich
square (in $\Sch$) after replacing $X, U, V, W$ by $\ol{X}, \ol{U},
\ol{V}, \ol{W}$ (\emph{cf.} Remark \ref{rk-graph-trick}(2)) and all
morphisms are minimal.
\end{enumerate}
\end{defn}

\begin{prop}[\protect{\cite[Proposition~3.2.3 (2)]{kmsy}}] \label{p3.1}
The cd-structure $P_{\ulMVfin}$ on $\ulMSm^\fin$ consisting of
$\ulMVfin$-squares  is strongly complete and strongly regular, hence
complete and regular in the sense of \cite{cdstructures} (see
\cite[Definition~A.7.1 and A.7.4]{kmsy}).
\end{prop}

We define $\ulMNS^{\fin}$ to be the full subcategory of
$\ulMPS^{\fin}$ consisting of sheaves with respect to the
Grothendieck topology associated to $P_{\ulMVfin}$. Let
\begin{equation}\label{eq;ulasfin}
\ul{a}^{\fin}_{s,\Nis} : \ulMPS^{\fin}\allowbreak \to \ulMNS^{\fin}
\end{equation}
be the sheafification functor, that is, the left adjoint of the
inclusion functor $\ul{i}^{\fin}_{s,\Nis} : \ulMNS^{\fin}
\hookrightarrow \ulMPS^{\fin}$. It exists for general reasons and is
exact \cite[expos\'e~II, th\'eor\`eme~3.4]{SGA4}.

\begin{lemma}[\protect{\cite[Lemma 3.1.3]{kmsy}}]\label{lem:equiv-smallsites}
Let $M \in \ulMP^{\fin}$. Let $M_\Nis$ be the category of morphisms
$f : N \to M$ in $\ulMP^{\fin}$ such that $\ol{f}$ is \'etale and
$\ol{f}^*M^\infty = N^\infty$, endowed with the topology induced by
$P_{\ulMVfin}$, and let $(\ol{M})_\Nis$ be the $($usual\,$)$ small
Nisnevich site on $\ol{M}$. Then we have an isomorphism of sites
\[ M_\Nis \to (\ol{M})_\Nis, \qquad N \mapsto \ol{N},
\]
whose inverse is given by $(p : X \to \ol{M}) \mapsto (X,
p^*(M^\infty))$. $($This isomorphism of sites depends on the choice of
$M^\infty.)$ \qed
\end{lemma}

\begin{nota}\label{n3.7a}
Let $M=(\ol{M},M^\infty) \in \ulMP^\fin$ and $F \in \ulMPS^\fin$. We
write $F_M$ for the presheaf on $(\ol{M})_\et$ which associates
$F(U,U\times_{\ol{M}} M^\infty)$ to an \'etale map $U \to \ol{M}$.
By Lemma \ref{lem:equiv-smallsites}, $F\in \ulMNS^\fin$ if and only
if $F_M$ is a sheaf on $(\ol{M})_\Nis$ for every $M\in \ulMP^\fin$.
\end{nota}

Let $\ulMNST^\fin$ be the full subcategory of $\ulMPST^\fin$
consisting of all objects $F \in \ulMPST^\fin$ such that
$\ul{c}^{\fin *} F \in \ulMNS^\fin$, where $\ul{c}^{\fin
*}:\ulMPST^\fin\to \ulMPS^\fin$ is from \eqref{MPScategories}.

We write $\ul{i}_{ \Nis}^{\fin}: \ulMNST^\fin\to \ulMPST^\fin$ for
the inclusion functor and $\ul{c}^{\fin \Nis} : \ulMNST \to \ulMNS$
for the functor induced by $\ul{c}^{\fin *}$. By definition, we have
\begin{equation}\label{eq:c-i-fin}
\ul{c}^{\fin *} \ul{i}_{\Nis}^{\fin} =\ul{i}_{s, \Nis}^{\fin}
\ul{c}^{\fin \Nis}.
\end{equation}

\begin{thm}[\protect{\cite[Theorem~3.5.3]{kmsy}}]\label{thm:sheaf-transfer}
The functor $\ul{i}_{ \Nis}^{\fin}$ has an exact left adjoint
\[\ul{a}_{ \Nis}^{\fin} : \ulMPST^{\fin} \to \ulMNST^{\fin}\]
satisfying
\begin{equation}\label{eq:c-a-fin}
\ul{c}^{\fin \Nis} \ul{a}_{ \Nis}^{\fin} = \ul{a}_{s \Nis}^{\fin}
\ul{c}^{\fin *}.
\end{equation}
In particular $\ulMNST^{\fin}$ is Grothendieck. Moreover,
$\ulMNST^{\fin}$ is closed under infinite direct sums in
$\ulMPST^{\fin}$ and the inclusion functor $\ul{i}_{\Nis}^{\fin} :
\ulMNST^{\fin} \to \ulMPST^{\fin}$  is strongly additive.
\end{thm}

\subsection{The $\protect\ulMV$ cd-structure}

\begin{defn}[\protect{\cite[Definition~4.1.1]{kmsy}}] \label{d4.1}
A diagram in $\ulMSm$ is called an \emph{$\ulMV$-square} if it is
isomorphic in $\ulMSm^\Sq$ to $\ul{b}_s^\Sq(Q)$ for some
$\ulMVfin$-square $Q$ as in Definition \ref{d3.2}, where $\ul{b}_s^\Sq$
is the functor induced on squares by $\ul{b}_s: \ulMSm^\fin \to
\ulMSm$ from \eqref{MPcategories}. Let $P_{\ul{\MV}}$ be a
cd-structure on $\ulMSm$ given by the collection of $\ulMV$-squares.
\end{defn}

\begin{thm}[\protect{\cite[Theorem~4.1.2]{kmsy}}]\label{thm;cd-str-ulMSm}
The cd-structure $P_{\ul{\MV}}$ is strongly complete and strong\-ly
regular, in particular complete and regular $($see \cite[Definition~A.7.1
and A.7.4]{kmsy}$)$.
\end{thm}

\begin{rk} In view of Lemma \ref{lem:equiv-smallsites}, the topology defined by $P_{\ul{\MV}^\fin}$ is subcanonical. This is also true for $P_{\ul{\MV}}$ by \cite[Theorem~4.5.1]{nistopmod}.
\end{rk}

We write $\ulMNS$ for the full subcategory of $\ulMPS$ consisting of
sheaves with respect to the Grothendieck topology on $\ulMSm$
associated to $P_{\ul{\MV}}$. We denote by $\ul{i}_{s,\Nis}: \ulMNS
\to \ulMPS$ the inclusion functor.

\begin{lemma}\label{lem:sheaves-MNS}
The functor $\ul{i}_{s, \Nis}: \ulMNS \to \ulMPS$ has an exact left
adjoint $\ul{a}_{s, \Nis}$. In particular $\ulMNS$ is Grothendieck.
Moreover, the following conditions are equivalent for $F \in
\ulMPS$.
\begin{thlist}
\item
$F\in \ulMNS$.
\item
It transforms any $\ulMVfin$-square
\begin{equation}\label{Q0} Q : \vcenter{\xymatrix{
W \ar[r]^{} \ar[d]_{} & V \ar[d]^{} \\
U \ar[r]^{} & M }}\end{equation} into an exact sequence
\[ 0\ \to F(M) \to F(U) \oplus F(V) \to F(W).\]
\end{thlist}
\end{lemma}
\begin{proof}
The first two assertions follow from the general properties of
Grothendieck topologies \cite[expos\'e~II]{SGA4}. The equivalence (i)
$\iff$ (ii) follows from \cite[Corollary~2.17]{cdstructures} in view of
Theorem \ref{thm;cd-str-ulMSm}.
\end{proof}

\begin{lemdef}[\protect{\cite[Lemma 4.5.1]{kmsy}}]\label{lem:mnst-condition}
For $F \in \ulMPST$, one has $\ul{c}^*F \in \ulMNS$ if and only if
$\ul{b}^*F \in \ulMNST^\fin$, where
\[\ul{b}^*: \ulMPST \to \ulMPST^\fin,\quad \ul{c}^*: \ulMPST \to \ulMPS\]
are from \eqref{MPScategories}.\\
We define $\ulMNST$ to be the full subcategory of $\ulMPST$
consisting of those $F$ enjoying these equivalent conditions. We
denote by $\ul{i}_{\Nis}: \ulMNST \to \ulMPST$ the inclusion
functor, and by $\ul{b}^\Nis : \ulMNST \to \ulMNST^\fin$ the functor
induced by $\ul{b}^*$.
\end{lemdef}

Recall the functor $\ul{b}_!: \ulMPST^\fin\to \ulMPST$ from Lemma
\ref{eq:bruno-functor}.

\begin{prop}[\protect{\cite[Proposition~4.5.4]{kmsy}}]\label{lem;b!ulMNST}
We have $\ul{b}_{!}(\ulMNST^\fin) \subset  \ulMNST$. We can thus consider
$\ul{b}_{\Nis} : \ulMNST^\fin \to \ulMNST$, the restriction of
$\ul{b}_{!}$, so that we have
\begin{equation}\label{eq:b-and-i3}
\ul{b}_{!} \ul{i}_{\Nis}^\fin = \ul{i}_{\Nis} \ul{b}_{\Nis}.
\end{equation}
Then $\ul{b}_{\Nis}$ is an exact left adjoint of $\ul{b}^\Nis$, and
is fully faithful.
\end{prop}

\begin{thm}[\protect{\cite[Lemma 4.5.3, Theorem~4.5.5 and Proposition~4.5.6]{kmsy}}]\label{thm:sheafification-ulMNST}
The category $\ulMNST$ contains $\Z_\tr(M)$ for any $M\in \ulMCor$.
The inclusion functor $\ul{i}_{\Nis}: \ulMNST \to \ulMPST$ has an
exact left adjoint
\[\ul{a}_{ \Nis} : \ulMPST \to \ulMNST\]
given by $\ul{a}_\Nis=\ul{b}_\Nis \ul{a}_\Nis^\fin \ul{b}^*$. In
particular, $\ulMNST$ is Grothendieck. Moreover, $\ulMNST$ is closed
under infinite direct sums in  $\ulMPST$, and $\ul{i}_\Nis$ is
strongly additive. We have
\begin{equation}\label{eq:a-c2}
\ul{b}_\Nis \ul{a}_\Nis^\fin = \ul{a}_\Nis \ul{b}_!, \qquad
\ul{a}_{s, \Nis} \ul{c}^* = \ul{c}^\Nis \ul{a}_\Nis,
\end{equation}
where $\ul{c}^\Nis$ is the functor determined by Lemma-Definition
\ref{lem:mnst-condition}. This functor is exact, strongly additive
and has a left adjoint $\ul{c}_\Nis=\ul{a}_{\Nis}\ul{c}_! \ul{i}_{s,
\Nis}$.
\end{thm}

\begin{nota}\label{not4}
Let $M\in \ulMCor$ and $F \in \ulMNST$. Using Notation \ref{n3.7a},
we define $F_M :=(\ul{b}^\Nis F)_M$, which is a sheaf on
$(\ol{M})_\Nis$.
\end{nota}

\begin{prop}[\protect{\cite[Proposition~4.6.3]{kmsy}}]\label{c3.1v}
Let $F\in \ulMNST$, and let $M\in \ulMCor$. Then there is a
canonical isomorphism for any $i\ge 0$:
\[\Ext^i_{\ulMNST}(\Z_\tr(M),F)\simeq
\colim_{N\in \ulSigma^{\fin}\downarrow M}H^i_\Nis(\Nb,F_N).
\]
Moreover, we have
\[
\colim_{N\in \ul{\Sigma}^{\fin}\downarrow M}
H_\Nis^i(\ol{N},(R^q(\ul{b}_s^\Nis)\ul{c}^\Nis F)_N)=0\text{ for all
$q>0$.}
\]
\end{prop}

\begin{cor}\label{c4.1} We have $\Ext^i_{\ulMNST}(\Z_\tr(M),F)=0$ for $i>\dim \ol{M}$.
\end{cor}

\begin{proof} For any $N\in \ulSigma^{\fin}\downarrow M$, we have $\dim \ol{N}=\dim \ol{M}$. Therefore the statement follows from Proposition \ref{c3.1v} and the known bound for Nisnevich cohomological dimension.
\end{proof}

\section{A cd-structure on $\protect\MSm$}

In this section we introduce a cd-structure on $\MSm$ and describe
its main properties, following the works of Miyazaki
\cite{nistopmod} and Kahn-Miyazaki \cite{KaMi}. For this we need to
start with the ``off-diagonal'' functor.

\subsection{Off-diagonal}

\begin{defn}
Define $\MEt$ as the category such that
\begin{enumerate}
\item objects are  those morphisms $f : M \to N$ in $\MSm$ such that $f^\o : M^\o \to N^\o$ is \'etale, and
\item a morphism from $f : M_1 \to N_1$ to $g : M_2 \to N_2$ is a pair of morphisms $(s: M_1 \to M_2 , t:N_1 \to N_2)$ in $\MSm$ which commute with $f,g$ and such that $s^\o$ and $t^\o$ are \emph{open immersions}.
\end{enumerate}
\end{defn}

For modulus pairs $M$ and $N$, we define \emph{the disjoint union of
$M$ and $N$} by
\[
M \sqcup N := (\ol{M} \sqcup \ol{N}, M^\infty \sqcup N^\infty ).
\]
Obviously, we have $(M \sqcup N)^\o = M^\o \sqcup N^\o$.

\begin{thm}[\protect{\cite[Theorem~3.1.3]{nistopmod}}]\label{thm:def-OD}
There exists a functor
\[
\OD : \MEt \to \MSm
\]
such that for any $f :  M \to N$, one has a functorial decomposition
\[
M \times_N M \cong M \sqcup \OD(f).
\]
Moreover, we have $\OD (f)^\o = M^\o \times_{N^\o} M^\o \setminus
\Delta (M^\o)$, where $\Delta : M^\o \to M^\o \times_{N^\o} M^\o$ is
the diagonal morphism. In particular, if $f^\o$ is an open
immersion, then $\OD (f)^\o = \emptyset$, hence $\OD (f) =
\emptyset$. We call the functors \emph{the off-diagonal functors}.
\end{thm}

\subsection{The $\protect\MV$ cd-structure}

\begin{defn}\label{def:new-MV}
Let $T$ be an object of $\MSm^\Sq$ of the form
\begin{equation}\label{Q}
\vcenter{ \xymatrix{
T(00) \ar[r]^{v_T} \ar[d]_{q_T} & T(01) \ar[d]^{p_T} \\
T(10) \ar[r]^{u_T} & T(11). } }\end{equation} Then $T$ is called an
\emph{$\MV$-square} if the following conditions hold:
\begin{enumerate}
\item \label{thm:new-MV1} $T$ is a pull-back square in $\MSm$.
\item \label{thm:new-MV2} There exist an $\ulMV$-square $S$ (\emph{cf.} Definition \ref{d4.1}) such that $S(11) \in \MSm$, and a morphism $S \to T$ in $\ulMSm^\Sq$ such that the induced morphism $S^\o \to T^\o$ is an isomorphism in $\Sm^\Sq$ and $S(11) \to T(11)$ is an isomorphism in $\MSm$.
In particular, $T^\o$ is an elementary Nisnevich square.
\item \label{thm:new-MV3} $\OD(q_T) \to \OD (p_T)$ is an isomorphism in $\MSm$.
\end{enumerate}

We let $P_{\MV}$ be the cd-structure on $\MSm$ consisting of
$\MV$-squares.
\end{defn}

The following are the main results of \cite{nistopmod}.

\begin{thm}[{\cite[{Theorem~4.3.1, 4.4.1 and 4.5.1}]{nistopmod}}]\label{cd-str-MSm}
The cd-structure $P_{\MV}$ is complete and regular. The associated
topology is subcanonical.
\end{thm}

\begin{cor}\label{l5.1.17} Define $\MNS \subset \MPS$ as the full subcategory consisting of sheaves with respect to the Grothendieck topology on $\MSm$ associated with $P_{\MV}$. Then there exists a pair of adjoint functors
\begin{equation}\label{eq;adjoint-MNS}
\MPS\begin{smallmatrix}\asNis \\\longrightarrow\\\ i_{s,\Nis}\\
\longleftarrow\end{smallmatrix}\MNS,
\end{equation}
where $i_{s,\Nis}$ is the natural inclusion and its left adjoint
$\asNis$ is exact. Moreover, $\MNS$ is Grothendieck. For $F \in
\MPS$, the following conditions are equivalent.
\begin{itemize}
\item[\rm (1)] $F \in \MNS$.
\item[\rm (2)]
For any $\MV$-square $Q$ as \eqref{Q}, the associated sequence
\begin{equation}\label{F(Q)}
 0\ \to F(T(00)) \to F(T(10)) \times F(T(01)) \to F(T(11))
\end{equation}
is exact.
\end{itemize}
\end{cor}

\begin{proof} Same as for Lemma \ref{lem:sheaves-MNS}:
the first two assertions are general facts on Grothendieck
topologies, and the equivalence (i) $\iff$ (ii) follows from
\cite[Corollary~2.17]{cdstructures} in view of Theorem \ref{cd-str-MSm}
(1).
\end{proof}

The following are the main results of \cite{KaMi}. To state them, we
need a definition.

\begin{defn}\label{d5.2}\
\begin{enumerate}
\item For any square $S \in \ulMSm^\Sq$, we define categories $\Comp(S)$ as the full subcategories of $S \downarrow \MSm^\Sq$ consisting of those objects $S \to T$ such that $S (ij) \to T(ij)$ belongs to $\Comp(S(ij))$ for any $(ij) \in \Sq$.
Here, $\Comp(-)$ is from Definition \ref{d1.12}.
\item
For an $\ulMVfin$-square $S$ in $\ulMSm^{\fin}$, an object $S \to
\tau^{\Sq}(Q)$ in $\Comp(S)$ is an $\MV$-completion of $S$ if $Q$ is
an $\MV$-square. We write
\[\Comp(S)^{\MV}\subset \Comp(S) \]
for the full subcategory consisting of $\MV$-completions of $S$.
\end{enumerate}
\end{defn}

\begin{thm}\label{t3.1} The following assertions hold.
\begin{itemize}
\item[\rm (1)]  \cite[Theorem~1]{KaMi}. The functor $\tau_s:\MSm\to \ulMSm$ is continuous in the sense of \cite[expos\'e~III]{SGA4} for the topologies given by $P_{\MV}$ and $P_{\ulMV}$.
\item[\rm (2)] \cite[Theorem~1.5.6]{KaMi}. For any $\ulMVfin$-square $S$ such that $\ol{S}(11)$ is normal, $\Comp(S)^{\MV}$ is cofinal in $\Comp(S)$.
\end{itemize}
\end{thm}

\begin{corollary}\label{edgewise-cofinality}
Let $S$ be an $\ulMVfin$-square  such that $\ol{S}(11)$ is normal.
Then, for any $i,j \in \{0,1\}$, the subcategory of $\Comp (S(ij))$
defined by \[\{T(ij)\ |\  T \in \Comp (S)^{\MV}\}\] is cofinal in
$\Comp (S(ij))$.
\end{corollary}

\begin{proof}
By Theorem \ref{t3.1}\,(1), it suffices to prove that the
subcateogry $\{T(ij)\ |\  T \in \Comp (S)\}$ is cofinal in $\Comp
(S(ij))$ for any $i,j \in \{0,1\}$. To show this we need the
following

\begin{lemma}
For any morphism $f : V \to U$ in $\ulMSm^\fin$ and for any $M \in
\Comp (U)$, there exists $N \in \Comp (V)$ such that $f$ induces a
morphism $N \to M$ in $\MSm^\fin$.
\end{lemma}

\begin{proof}
Take any $N \in \Comp (V)$. Let $\Gamma$ be the graph of the
rational map $\ol{N} \dashrightarrow \ol{M}$ and let $\pi : \Gamma
\to \ol{N}$. Then $\pi$ induces an isomorphism $N' := (\Gamma ,
\pi^\ast N^\infty) \cong N$ and $N' \in \Comp (V)$. Therefore, by
replacing $N$ with $N'$, we may assume that $\ol{f} : \ol{U} \to
\ol{V}$ extends to a morphism of schemes $\ol{p} : \ol{N} \to
\ol{M}$. Moreover, by taking blow up $\Bl_{\ol{N} - \ol{V}} (\ol{N})
\to \ol{N}$ and pulling back the divisor, we may assume that there
exists an effective Cartier divisor $D$ on $\ol{N}$ with $\ol{N} -
\ol{V} = |D|$. Since the admissiblity of $V \to U$ implies
$\ol{p}^\ast M^\infty |_{\ol{V}^N} = \ol{f}^\ast U^\infty
|_{\ol{V}^N} \leq V^\infty |_{\ol{V}^N} = N^\infty |_{\ol{V}^N}$,
\cite[Lemma 3.14]{MiCI} shows that there exists a positive integer
$n$ such that $\ol{p}^\ast M^\infty |_{\ol{N}^N} \leq (N^\infty +
nD) |_{\ol{N}^N}$. Then $N'':=(\ol{N}, N^\infty + nD) \in \Comp
(V)$, it dominates $N$ and $\ol{p}$ induces a morphism $N'' \to M$
in $\ulMSm^\fin$, as desired.
\end{proof}

The corollary immediately follows from the lemma when $(i,j) =
(1,1)$. We prove the case $(i,j)=(1,0)$. Take any $N \in \Comp
(S(10))$ and any $T \in \Comp (S)$. Since $\Comp (S(10))$ is
filtered, there exists $T'(10) \in \Comp (S(10))$ which dominates
both $N$ and $T(10)$. By the lemma there exists $T'(00) \in \Comp
(S(00))$ such that $S(00) \to S(10)$ extends to $T'(00) \to T(10)$.
Since $\Comp (S(00))$ is ordered we may assume that $T'(00)$
dominates $T(00)$. Then the resulting diagram
\[\xymatrix{
T'(00) \ar[r] \ar[d] & T(01) \ar[d] \\
T'(10) \ar[r] & T(11) }\] is an object of $\Comp (S)$ dominating
$T$, and $T'(10)$ dominates $N$ by construction. This proves the
case $(i,j)=(1,0)$. The proof for $(i,j)=(0,1)$ is completely the
same.

Finally we prove the case $(i,j)=(0,0)$. Take any $N \in \Comp
(S(00))$ and take any $T \in \Comp (S)$. Since $\Comp (S(00))$ is
filtered, there exists $T'(00) \in \Comp (S(00))$ which dominates
both $N$ and $T(00)$. Then the square obtained by replacing $T(00)$
with $T'(00)$ dominates $T$. This finishes the proof of Corollary
\ref{edgewise-cofinality}.
\end{proof}

\begin{remark}
The essential point of the above proof is the fact that the diagram
category $\Sq$ does not have a loop, and therefore the use of the
graph trick terminates in finitely many steps. We remark that we can
generalize the proof to a much more abstract argument, \emph{cf.}
\cite[Lemma C.6]{KaMi0}.
\end{remark}

\section{Sheaves on $\MSm$ and $\MCor$}

\subsection{Sheaves on $\MSm$}\label{s4.1}

\begin{lemma}\label{lcom10} The category $\MNS$ is closed under infinite direct sums and the inclusion functor
$i_{s, \Nis}$ is strongly additive.
\end{lemma}
\begin{proof} Indeed, the sheaf condition is tested on finite diagrams, hence the presheaf given by a direct sum of sheaves is a sheaf.
\end{proof}

\begin{thm}\label{thm:MNS-tau-a}
The following assertions hold.
\begin{itemize}
\item[\rm (1)]
We have $\tau_s^*(\ulMNS) \subset \MNS$ and $\tau_{s !}(\MNS)
\subset \ulMNS$, where $\tau_s^*$ and $\tau_{s !}$ are from Lemma
\ref{eq.tau}. Hence we obtain functors
\[\tau_s^\Nis : \ulMNS \to \MNS \qaq \tau_{s, \Nis} : \MNS \to \ulMNS\]
such that
\begin{equation}\label{eq:taus-and-is}
i_{s, \Nis} \tau_s^\Nis = \tau_s^* \ul{i}_{s, \Nis}, \quad \tau_{s
!} i_{s, \Nis} = \ul{i}_{s, \Nis} \tau_{s, \Nis},
\end{equation}
and that $\tau_s^\Nis$ is right adjoint to $\tau_{s, \Nis}$. $($See
\eqref{MPcategories} for $\tau_s$.$)$
\item[\rm (2)]
For $F \in \MPS$, one has $F \in \MNS$ if and only if $\tau_{s,!}F
\in \ulMNS$.
\item[\rm (3)]
We have
\begin{align}\label{eq:a-tau-c}
&\ul{a}_{s, \Nis} \tau_{s !} = \tau_{s, \Nis} a_{s, \Nis},
\end{align}
\item[\rm (4)]
The functor  $\tau_{s, \Nis}$ is fully faithful and exact. Moreover,
the functor $\tau_s^\Nis$ preserves injectives.
\end{itemize}
\end{thm}

\begin{proof}
First we prove (1). The first assertion follows from the continuity
of $\tau_s$ (Theorem \ref{t3.1}\,(1)). Similarly, the second
assertion morally follows from the ``continuity of $\tau_s^!$'' (see
\cite[Remark~1]{KaMi}): we give a proof based on Theorem
\ref{t3.1}\,(2).

Take $F\in \MNS$. By Lemma \ref{lem:sheaves-MNS}, it suffices to
show that the sequence
\[ 0\to \tau_!F(M) \to \tau_!F(U)\times \tau_!F(V) \to \tau_!F(W)
\]
is exact for any $Q\in P_{\ulMV}$ as in \eqref{Q0}. By Remark
\ref{rk-graph-trick} (1), we may assume that $\ol{M}$ is normal.
Since a filtered colimit of exact sequences of abelian groups is
exact, the desired assertion follows from Lemma \ref{eq.tau} (4),
Corollary \ref{l5.1.17} and  Corollary \ref{edgewise-cofinality}.
Finally the adjointness follows from Lemma \ref{eq.tau} (2). This
completes the proof of (1).

(2) follows from (1) and $\id \cong \tau_s^* \tau_{s,!}$ by Lemma
\ref{eq.tau} (2).

(3) follows from \eqref{eq:taus-and-is} by adjunction.

In (4), the exactness of $\tau_{s, \Nis}$ follows from Lemma
\ref{lem:lr-adjoint} (4) applied to $c=\tau_{s, !}$, using (1), (3)
and Lemma \ref{lem:sheaves-MNS}. It then implies the preservation of
injectives by  $\tau_s^\Nis$. The full faithfulness of $\tau_{s,
\Nis}$ follows from that of $\tau_{s!}$ (Lemma \ref{eq.tau} (2)) and
\eqref{eq:taus-and-is}.
\end{proof}

\begin{rk}
By \cite[expos\'e~II, Proposition~1.3]{SGA4}, the continuity of $\tau_s$ implies
\emph{a priori} the existence of a left adjoint $\tau_{s,\Nis}$ to
$\tau_s^\Nis$ on sheaves of sets, having the following properties:
\begin{enumerate}
\item \eqref{eq:a-tau-c};
\item $\tau_{s,\Nis}=a_{s,\Nis} \tau_{s,!} i_{s,\Nis}$;
\item $\tau_{s,\Nis}$ commutes with representable colimits;
\item $\tau_{s,\Nis}$ sends a representable sheaf to the corresponding representable sheaf.
\end{enumerate}

All this extends to sheaves of abelian groups by \cite[expos\'e~II, Proposition~6.3.1]{SGA4}. The one property which is missing is the right formula
of \eqref{eq:taus-and-is}. From \eqref{eq:a-tau-c}, we deduce a base
change morphism
\begin{equation}\label{eq4.1}
\tau_{s !} i_{s, \Nis} \Rightarrow \ul{i}_{s, \Nis} \tau_{s, \Nis}
\end{equation}
that we have shown above to be an isomorphism. By (4), this
isomorphism is clear on the generators $\Z[M]$ ($M\in \MSm$) of the
Grothendieck category $\MNS$, since both topologies on $\MSm$ and
$\ulMSm$ are subcanonical (Theorem~\ref{thm;cd-str-ulMSm} and
\ref{cd-str-MSm} (1)). But both sides of \eqref{eq4.1} are a
composition of left and right exact functors, so this is not
sufficient to get the general case. Thus the recourse to Theorem
\ref{t3.1}\,(2) seems necessary to prove Theorem \ref{thm:MNS-tau-a}
(1).
\end{rk}

\subsection{Sheaves on $\protect\MCor$} The following is an analogue of Lemma-Definition \ref{lem:mnst-condition}:

\begin{lemdef}\label{lem:mnst-condition2}
For $F \in \MPST$, one has
$\tau_!F \in \ulMNST$ if and only if $c^*F \in \MNS$. $($See \eqref{MPcategories} for $c$.$)$
We write $\MNST$ for the full subcategory of $\MPST$ consisting of
those $F \in \MPST$ that enjoy these equivalent conditions. Let
$i_\Nis : \MNST \to \MPST$ be the inclusion functor and let
$\tau_\Nis : \MNST \to \ulMNST$ and $c^\Nis : \MNST \to \MNS$ be
such that
\begin{equation}\label{eq:i-tau-sh}
\ul{i}_\Nis \tau_\Nis = \tau_! i_\Nis, \quad i_{s, \Nis} c^\Nis =
c^* i_\Nis.
\end{equation}
\end{lemdef}

\begin{proof}Let $F\in \MPST$. Then
\[\tau_! F\in \ulMNST\Rightarrow \tau_{s!} c^*F = \ul{c}^*\tau_! F\in c^*\ulMNST \subset \ulMNS\]
where we used \eqref{eq:c-and-tau} and Lemma-Definition
\ref{lem:mnst-condition}; by Theorem \ref{thm:MNS-tau-a} (2), this
implies $c^*F\in \MNS$. This reasoning can be reversed.
\end{proof}

\begin{lemma}\label{l4.2}\
\begin{itemize}
\item[\rm (1)]
For any $M\in \MCor$, the presheaf $\Z_\tr(M)\in \MPST$ belongs to
$\MNST$.
\item[\rm (2)] We have a natural isomorphism
\begin{equation}\label{eq:i-and-c-underlines}
\tau_{s, \Nis} c^\Nis  \cong \ul{c}^\Nis \tau_\Nis.
\end{equation}
\end{itemize}
\end{lemma}
\begin{proof}
(1) follows from Theorem \ref{thm:sheafification-ulMNST}, since
$\tau_!\Z_\tr(M)=\Z_\tr(\tau M)$. (2) is a consequence of
\eqref{eq:i-tau-sh} and the corresponding isomorphism for presheaves
\eqref{eq:c-and-tau}.
\end{proof}

\begin{lemma}\label{lem:patching2}
Let us consider the ``$2$-Cartesian product category'' $\MNS
\times_{\ulMNS} \ulMNST$, that is, the category of triples $(F_s,
F_t, \phi)$ consisting of $F_s \in \MNS, ~F_t \in \ulMNST$ and an
isomorphism $\phi : \tau_{s, \Nis} F_s \iso \ul{c}^\Nis F_t$ in
$\ulMNS$. The functor
\[ \MNST \to \MNS \times_{\ulMNS} \ulMNST,
\]
defined by $F \mapsto (c^\Nis F, \tau_\Nis F, \theta_F)$, where
$\theta_F : \tau_{s, \Nis} c^\Nis F \iso \ul{c}^\Nis \tau_\Nis F$ is
from \eqref{eq:i-and-c-underlines}, is an equivalence of categories.
\end{lemma}

\begin{proof}
The same statement was proven for presheaves in \cite[Lemma
2.7.1]{kmsy};  full faithfulness follows from this, and essential
surjectivity follows from Lemma-Definition
\ref{lem:mnst-condition2}.
\end{proof}

\begin{thm}\label{thm:a-nis-final}
The following assertions hold.
\begin{itemize}
\item[\rm (1)]
The functor $i_\Nis$ is strongly additive and has an exact left
adjoint $a_\Nis$. Consequently, $\MNST$ is Grothendieck. We have
\begin{equation}\label{eq:Anis-c-tau}
c^\Nis a_\Nis = a_{s, \Nis} c^*, \quad \tau_\Nis a_\Nis =\ul{a}_\Nis
\tau_!.
\end{equation}
\item[\rm (2)]
The functor $c^\Nis$ has a left adjoint $c_\Nis=a_\Nis c_! i_\Nis$.
Moreover, $c^\Nis$ is exact, strongly additive, and faithful.
\end{itemize}
\end{thm}

\begin{proof}
By Definition of $\MNST$, the strong addtivity of $i_\Nis$ follows
from that of $\uli_\Nis$ (Theorem \ref{thm:sheafification-ulMNST}
and Lemmma \ref{lem:lr-adjoint} (2)). We then use Lemma
\ref{lem:patching2} and \cite[Lemma 2.7.1]{kmsy} to construct
$a_\Nis$ by patching $a_{s,\Nis}$ and $\ul{a}_\Nis$ over
$\ul{a}_{s,\Nis}$, \emph{i.\,e.} we want $a_\Nis$ to verify
\eqref{eq:Anis-c-tau}; a formal argument shows that such a patching
is determined by the second isomorphism of \eqref{eq:a-c2} and by
the one of \eqref{eq:a-tau-c}.
\[
\xymatrix{ & \MPST \ar[dl]_{c^*} \ar[dr]^{\tau_!}
\ar@/^30ex/[ddd]_{a_\Nis} &
\\
\MPS \ar[r]_{\tau_{s!}} \ar[d]_{a_{s, \Nis}} & \ulMPS
\ar[d]^{\ul{a}_{s, \Nis}} & \ulMPST \ar[l]^{\ul{c}^*}
\ar[d]^{\ul{a}_{\Nis}}
\\
\MNS \ar[r]^{\tau_{s, \Nis}} & \ulMNS  & \ulMNST
\ar[l]_{\ul{c}^\Nis}
\\
& \MNST \ar[ul]^{c^\Nis} \ar[ur]_{\tau_\Nis}. & }
\]

The second isomorphism of \eqref{eq:Anis-c-tau} easily implies that
$a_\Nis$ is left adjoint to $i_\Nis$. Then (2) follows from Lemma
\ref{eq:c-functor}, Lemma \ref{lcom10} and Lemma
\ref{lem:lr-adjoint} (3).

Finally, the exactness of $a_\Nis$ is a consequence of the first
isomorphism of \eqref{eq:Anis-c-tau} since $c^\Nis$ is faithfully
exact as we have just shown.
\end{proof}

\begin{lemma}\label{lem:tau!-exact}
The following assertions hold.
\begin{itemize}
\item[\rm (1)]
We have $\tau^*(\ulMNST) \subset \MNST$.
\item[\rm (2)]
Let $\tau^\Nis : \ulMNST \to \MNST$ be the functor characterized by
\begin{equation}\label{eq:i-tau-star}
\tau^* \ul{i}_\Nis = i_\Nis \tau^\Nis.
\end{equation}
Then $\tau^\Nis$ is a right adjoint of $\tau_\Nis$, and
 $\tau_\Nis$ is fully faithful, exact, and strongly additive.
Moreover, $\tau^\Nis$ preserves injectives and is strongly additive.
\end{itemize}
\end{lemma}

\begin{remark}
We will see in Theorem \ref{thm;tauexact} below that $\tau^\Nis$ is
also exact.
\end{remark}

\begin{proof}
Let $F \in \ulMNST$ so that $\ul{c}^* F \in \ulMNS$. By Theorem
\ref{thm:MNS-tau-a} (1), we have $c^* \tau^* F = \tau_s^* \ul{c}^* F
\in \MNS$. In view of Lemma-Definition \ref{lem:mnst-condition2},
this proves that $\tau^* F \in \MNST$, whence (1).

In (2), the existence (and uniqueness) of $\tau^\Nis$ follows from
(1) (and the full faithfulness of $ \ul{i}_\Nis$ and $i_\Nis$). The
adjointness is shown by using the full faithfulness of $\ul{i}_\Nis$
and $i_\Nis$, the adjoint pair $(\tau_!, \tau^*)$,
\eqref{eq:i-tau-sh} and \eqref{eq:i-tau-star}. Similarly, the full
faithfulness of $\tau_\Nis$ follows from that of $\tau_!$
(Proposition \ref{eq.tau}) and \eqref{eq:i-tau-sh}. The  strong
additivity (resp. exactness)  of $\tau_\Nis$ follows from
Proposition \ref{eq.tau}, Theorem \ref{thm:a-nis-final} and Lemma
\ref{lem:lr-adjoint} (2) (resp. (4)), applied with $c=\tau_!$; the
latter implies that $\tau^\Nis$ preserves injectives. Finally, its
strong additivity is reduced to that of $\tau^*$ (Lemma \ref{eq.tau}
(1)),  $\ul{i}_\Nis$ (Theorem \ref{thm:sheafification-ulMNST}) and
$i_\Nis$ (Theorem \ref{thm:a-nis-final}) by the full faithfulness of
$i_\Nis$.
\end{proof}

\section{Cohomology in $\protect\MNST$}

\subsection{Main result} We begin with the following.

\begin{thm}\label{thm;tauexact}
The functor $\tau^\Nis : \ulMNST \to \MNST$ from Lemma
\ref{lem:tau!-exact} is exact.
\end{thm}

The proof will be given later in this section (see Corollary
\ref{tex1}). We now deduce its consequences.

\begin{lemma}\label{lem:coh-MNST-ulMNST}
For any $M \in \MCor, ~F \in \MNST, ~G \in \ulMNST$ and $q \ge 0$,
we have natural isomorphisms
\begin{align*}
&\Ext_{\MNST}^q(\Z_\tr(M), \tau^\Nis G) \cong
\Ext_{\ulMNST}^q(\Z_\tr(M), G),
\\
&\Ext_{\MNST}^q(\Z_\tr(M), F) \cong \Ext_{\ulMNST}^q(\Z_\tr(M),
\tau_{\Nis} F).
\end{align*}
\end{lemma}

\begin{proof}
By Theorem \ref{thm;tauexact}, $\tau^\Nis$ is exact and it preserves
injectives by Lemma \ref{lem:tau!-exact}. Hence we have
\[
R^q(i_{ \Nis}) \tau^\Nis G =R^q(i_{\Nis} \tau^\Nis) G =R^q(\tau^*
\ul{i}_{\Nis}) G =\tau^* R^q \ul{i}_{\Nis} G
\]
for any $G\in \ulMNST$ by \cite[Theorem~A.9.1]{kmsy}. Using the
projectivity of $\Z_\tr(M)$ in $\MPST$ and that of $\tau_!
\Z_\tr(M)=\Z_\tr(M)$ in $\ulMPST$, and using
Proposition \ref{pA.2} twice, we get isomorphisms
\begin{multline*}
\Ext_{\MNST}^q(a_\Nis\Z_\tr(M), \tau^\Nis G)\simeq \MPST(\Z_\tr(M),R^q(i_{ \Nis}) \tau^\Nis G)\\
\simeq \MPST(\Z_\tr(M),\tau^* R^q \ul{i}_{\Nis} G)\simeq \ulMPST(\tau_!\Z_\tr(M),R^q \ul{i}_{\Nis} G)\\
\simeq \ulMPST(\Z_\tr(M),R^q \ul{i}_{\Nis} G)\simeq
\Ext_{\ulMNST}^q(\ul{a}_\Nis\Z_\tr(M),G).
\end{multline*}
Moreover, $a_\Nis \Z_\tr(M) = \Z_\tr(M)$ and
$\ul{a}_\Nis\Z_\tr(M)=\Z_\tr(M)$ by Theorem
\ref{thm:sheafification-ulMNST} and Lemma \ref{l4.2} (1), whence the
first formula by evaluating both sides at $M$. The second one
follows from the first by taking $G=\tau_{\Nis}F$, since
$\tau^\Nis\tau_\Nis=\id$.
\end{proof}

Let $M\in \MCor$ and $F\in \MNST$. Using Notation \ref{not4}, we
define $F_M :=(\tau_\Nis F)_M$, which is a sheaf on $(\ol{M})_\Nis$.

\begin{thm}\label{thm:coh-MNST}
For any $p\ge 0$, $M\in \MCor$ and $F\in \MNST$, we have a natural
isomorphism
\begin{equation} \label{eq:ext-coh-last}
 \Ext_{\MNST}^p(\Z_{\tr}(M), F) \simeq
\colim_{N\in \Sigma^{\fin}\downarrow M} H_\Nis^p(\ol{N},F_N).
\end{equation}
Moreover, we have
\[
\colim_{N\in \ul{\Sigma}^{\fin}\downarrow M}
H_\Nis^p(\ol{N},(R^q(\ul{b}_s^\Nis)\ul{c}^\Nis\tau_\Nis
F)_N)=0\text{ for all $q>0$.}
\]
\end{thm}
\begin{proof}
Combine Proposition \ref{c3.1v}, Lemma \ref{lem:tau!-exact} and
Lemma \ref{lem:coh-MNST-ulMNST}.
\end{proof}

\begin{cor}\label{c5.1} We have $\Ext^q_{\MNST}(\Z_\tr(M),F)=0$ for $q>\dim \ol{M}$.
\end{cor}

\begin{proof} Same as for Corollary \ref{c4.1}. \end{proof}

\subsection{A generation lemma}
We now start proving Theorem \ref{thm;tauexact}. We need some
preliminaries.

\begin{lemma}\label{lem;uaNisvanish}
Let $F\in \ulMPST$ such that $\ulaNis F=0$. Then $F$ may be written
as a quotient of a direct sum of
\[ \Ztr(M/U):=\Coker\big( \Ztr(U) \to \Ztr(M\big)),\]
where $M\in \ulMSm$, $U\to M$ is a covering for the Grothendieck
topology on $\ulMSm^\fin$ associated to $P_{\ulMVfin}$ from
Proposition \ref{p3.1}, and the cokernel is taken in $\ulMPST$.
Moreover we have $\ulaNis \Ztr(M/U)=0$.
\end{lemma}
\begin{proof}
Let $G=c^* F\in \ulMNS$. Take $f\in G(L)=F(L)$ for $L\in \ulMSm$. By
\eqref{eq:a-c2} we have $\ulasNis G= 0$. By \cite[Lemma
4.3.2]{kmsy}, we have $\phi^*f=0$ for a cover $\phi: U\to M \to L$
in $\ulMSm_\Nis$, where $U\to M$ is a strict Nisnevich cover and
$M\to L$ is in $\ul{\Sigma}^\fin$. Hence the Yoneda map $\Zp(L) \to
G$ in $\ulMPS$ given by $f$ (see Definition \ref{d2.7} (4) for
$\Z^p(L)$) factors through
\[\Zp(L/U) :=\Coker( \Zp(U) \to\Zp(L))\]
(see Definition \ref{d2.7} (4) for $\Zp(-)$). By the adjunction
$(c_!,c^*)$ this induces a map $c_!\Zp(L/U) \to F$, and
\[ c_!\Zp(L/U) \simeq \Coker(c_!\Zp(U) \to c_!\Zp(L)) \simeq
\Coker(\Ztr(U) \to \Ztr(L)) ,\] where the first isomorphism follows
from the right exactness of $c_!$ as a left adjoint. This implies
that the Yoneda map $y(f):\Ztr(L) \to F$ in $\ulMPST$ factors
through the cokernel of $\Ztr(U) \to \Ztr(M) \to \Ztr(L)$. Thus we
get an induced map $\ol{y(f)}: \Ztr(M/U) \to F$. Since the map
$\Ztr(M) \to \Ztr(L)$ is an isomorphism in $\ulMPST$, the image of
$y(f)$ coincides with that of $\ol{y(f)}$. Collecting them over all
pairs $(L,f)$, this proves the first part of the lemma. Finally the last statement
follows from \cite[Theorem~4.5.7]{kmsy}.
\end{proof}

\subsection{The $\tau$ construction}

\begin{defn}\label{def1;MV}
Let $M, N \in \ulMCor$.
\begin{itemize}
\item[(1)]
We put
\[\Ztrtau M =\tau_!\tau^*\Ztr(M)\in \ulMPST.\]
Note $\Ztrtau M\in \ulMNST$ by Lemma \ref{lem:tau!-exact}.
\item[(2)]
Let $\ulMCortau(N,M)$ be the subgroup of $\ulMCor(N,M)$ generated by
elementary correspondences $Z$ in $\Cor(N^\circ,M^\circ)$ which
satisfy the condition:
\begin{enumerate}
\item[$(\spadesuit)$]
There exists a dense open immersion $j:\Nb\hookrightarrow \Lb$ with
$\Lb$ proper such that the closure $\ol{Z}$ of $Z$ in $\Lb\times
\Mb$ is proper over $\Lb$.
\end{enumerate}
\end{itemize}
\end{defn}

\begin{lemma}\label{lem1;MV}
For $N,M$ as above, the condition $(\spadesuit)$ is independent of
the choice of $j:\Nb\hookrightarrow \Lb$, and we have
\[\Ztrtau M(N)=\ulMCortau(N,M).\]
\end{lemma}

\begin{proof}
If $j':\Nb\hookrightarrow \Lb'$ is another choice equipped with
(proper) surjective $f: \Lb\to \Lb'$ such that $j'=f j$, writing
$\Zb'\subset \Lb'\times \Mb$ for the closure of $Z$, $f$ induces a
proper surjective map $\Zb'\to \Zb$. Then it is easy to see that
$\Zb'$ is proper over $\Lb'$ if and only if so is $\Zb$ over $\Lb$.
This proves the first assertion. To show the second assertion, we
note that by Lemma \ref{eq.tau} (3),
\[ \Ztrtau M(N) =\colim_{L\in \Comp(N)} \ulMCor(L,M).\]
The second assertion follows from this using the first assertion
(see the proof of \cite[Lemma 1.8.3]{kmsy}).
\end{proof}

\begin{lemma}\label{lem1.1;MV}
For $M, N \in \ulMCor$, we put
\[ \Ztrfintau M(N)=\ulMCor^{\fin}(N,M)\cap \ulMCortau(N,M) \subset \ulMCor(N,M).\]
Then $\Ztrfintau M$ defines an object of $\ulMPST^\fin$. Moreover we
have
\[\ul{b}_!\Ztrfintau M=\Ztrtau M.\]
\end{lemma}
\begin{proof}
The first assertion follows from Lemma \ref{lem1;MV}. For $N'\in
\ul{\Sigma}^\fin\downarrow N$ (see Definition \ref{def;ulSigmafin}),
Lemma \ref{lem1;MV} implies
\[ \ulMCortau(N,M) = \Ztrtau M (N)=\Ztrtau M (N')= \ulMCortau(N',M),\]
where the second equality follows from the fact that $N'\simeq N$ in
$\ulMCor$. By definition this implies
\[ \Ztrfintau M(N')= \ulMCor^\tau(N,M) \cap \ulMCor^\fin(N',M),\]
which proves the second assertion in view of the isomorphisms
\[ \ul{b}_!\Ztrfintau M(N) = \colim_{N'\in \ul{\Sigma}^\fin\downarrow N} \Ztrfintau M(N'),\]
\[ \ulMCor(N,M) = \colim_{N'\in \ul{\Sigma}^\fin\downarrow N} \ulMCor^\fin(N',M),\]
which hold by \eqref{eq:b-sh-explicit} and \cite[Proposition~1.9.2]{kmsy}.
\end{proof}

\begin{remark}
We can prove that $\Ztrfintau M$ lies in $\ulMNST^\fin$ (so that we
may remove $\ulaNisfin$ in Theorem \ref{thm1;MV}(1) below).
\end{remark}

\subsection{Exactness of a certain \v{C}ech complex}

\begin{thm}\label{thm1;MV}
Let $p : U \to M$ be a covering for the Grothendieck topology on
$\ulMSm^\fin$ associated to $P_{\ulMVfin}$ from Proposition
\ref{p3.1}. Denote by $U\times_M U$ the fiber product in
$\ulMSm^\fin$ $($see \cite[Proposition~1.10.4 and Corollary~1.10.7]{kmsy}$)$.
\begin{itemize}
\item[\rm (1)]
The \v{C}ech complex
 \[\dots \to \ulaNisfin \Ztrfintau{U \times_M U} \to \ulaNisfin\Ztrfintau{U}
 \to \ulaNisfin\Ztrfintau{M}  \to 0\]
is exact in $\ulMNST^\fin$.
\item[\rm (2)]
The \v{C}ech complex
 \[\dots \to \Ztrtau{U \times_M U}
 \to \Ztrtau{U}
 \to \Ztrtau{M}
 \to 0\]
is exact in $\ulMNST$.
\end{itemize}
\end{thm}

Theorem \ref{thm1;MV} (2) follows from (1) by applying the exact
functor $\ul{b}_\Nis$ from Proposition \ref{c3.1v} and using
isomorphisms
\[ \ul{b}_\Nis \ulaNisfin\Ztrfintau{M} \simeq \ulaNis \ul{b}_! \Ztrfintau{M} \simeq \ulaNis \Ztrtau{M} =\Ztrtau{M},\]
where the first isomorphism follows from \eqref{eq:a-c2}, the second
from Lemma \ref{lem1.1;MV} and the last equality follows from the
fact that $\Ztrtau M \in \ulMNST$ thanks to Lemma
\ref{lem:tau!-exact} (1).

We need some preliminaries for the proof of (1). It is inspired by
that of \cite[Theorem~3.4.1]{kmsy},  with some elaboration. Take
$(X,D)\in \ulMCor$ and a point $x\in X$. Let $\{\Xlam\}_{\lambda\in
\Lambda}$ be the filtered system of connected affine \'etale
neighborhoods of $x\in X$. Let
\begin{equation}\label{eq;olS}
\ol{S}=\lim_{\lambda\in \Lambda} \Xlam
\end{equation}
be the henselization of $X$ at $x$. Take $M\in \ulMCor$ and let
$\sD$ be the category of diagrams
\begin{equation}\label{eq0;MV}
\ol{S} \overset{f}{\leftarrow} Z \overset{g}{\to} \ol{M}
\end{equation}
of $k$-schemes with $f$ quasi-finite such that $Z\to \ol{S}\times
\ol{M}$ is a closed immersion and $V \not\subset \ol{S}\times
M^\infty$   for any irreducible component $V$ of $Z$. We denote
\eqref{eq0;MV} by $(Z,f,g)$. A morphism from $(Z,f,g)$ to
$(Z',f',g')$ is given by a morphism $\phi:Z\to Z'$ which fits into a
commutative diagram
\begin{equation}\label{eq3;MV}
\vcenter{ \xymatrix{
 &\ar[ld]_{f} Z \ar[rd]^{g}\ar[dd]^{\phi}\\
\ol{S} & & \ol{M}\\
&\ar[lu]^{f'} Z' \ar[ru]_{g'}. } }\end{equation}

Note that $\phi$ is automatically a closed immersion, so $\sD$ is a
cofiltered ordered set as it is stable under unions. For $(Z,f,g)\in
\sD$ let $E(Z)=E(Z,f,g)$ be the set of irreducible components $V$ of
$Z$ which belong to $\ulMCor^\fin((\ol{S},D),M)$, i.e. such that
$f|_V$ is finite and surjective over an irreducible component of
$\ol{S}$ and satisfies the admissibility condition:
\begin{equation}\label{eq3.5;MV}
(f \circ i_V \circ v)^*(D\times_X \ol{S}) \geq (g \circ i_V \circ
v)^*(M^\infty),
\end{equation}
where $v: V^N \to V$ is the normalization and $i_V : V
\hookrightarrow Z$ is the inclusion. Let $E^\tau(Z)\subset E(Z)$ be
the subset of those $V$ which belong to
$\Z_\tr^\fin(M)^\tau(\ol{S},D)$, i.e. satisfying the following
condition: there exists $\lambda\in \Lambda$ such that $(Z,f,g)$
(resp. $V\hookrightarrow Z$) is the base change via $\ol{S}\to
\Xlam$ of
\begin{equation}\label{eq1;MV}
\xymatrix{ \Xlam &\ar[l]_{f_\lambda}Z_\lambda \ar[r]^{g_\lambda} &
\ol{M}\\}\quad \text{(resp. $V_\lambda \hookrightarrow Z_\lambda$)},
\end{equation}
where $V_\lambda$ is an irreducible component of $Z_\lambda$
satisfying the condition:
\begin{enumerate}
\item[$(\clubsuit)_\lambda$]
$V_\lambda$ is finite over $\Xlam$ and satisfies the admissibility
condition
\begin{equation}\label{eq;admisiibilityVlam}
(f_\lambda \circ i_{V_\lambda} \circ v_\lambda)^*(D\times_X \Xlam)
\geq (g_\lambda \circ i_{V_\lambda} \circ v_\lambda)^*(M^\infty),
\end{equation}
similar to \eqref{eq3.5;MV}. Moreover, letting
$\tilde{V_\lambda}=h_\lambda(V_\lambda)$ with
$h_\lambda=(f_\lambda,g_\lambda):Z_\lambda\to \Xlam\times\ol{M}$
($\tilde{V_\lambda}$ is finite over $\Xlam$ by the finiteness of
$V_\lambda\to\Xlam$), there exists a dense open immersion
$\Xlam\hookrightarrow \ol{\Xlam}$ with $\ol{\Xlam}$ proper such that
the closure $\ol{\tilde{V_\lambda}}$ of $\tilde{V_\lambda}$ in
$\ol{\Xlam}\times\ol{M}$ is proper over $\ol{\Xlam}$.
\end{enumerate}

Let $L^\tau(Z)$ be the free abelian group on the set $E^\tau(Z)$.

\begin{lemma}\label{lem2;MV}
Let $V_\lambda$ be as in $(\clubsuit)_\lambda$ and $X_\mu\to \Xlam$
($\lambda,\mu\in \Lambda$) be a map in the system of \'etale
neighborhoods of $x\in X$. Let
\begin{equation}\label{eq2;MV}
\xymatrix{ \Xmu &\ar[l]_{f_\mu}Z_\mu \ar[r]^{g_\mu} & \ol{M}\\}\quad
\text{(resp. $V_\mu \hookrightarrow Z_\mu$)}
\end{equation}
be the base change of \eqref{eq1;MV} (resp. $V_\lambda
\hookrightarrow Z_\lambda$). If $V_\lambda\subset Z_\lambda$
satisfies $(\clubsuit)_\lambda$, then any component of $V_\mu$
satisfies $(\clubsuit)_\mu$.
\end{lemma}
\begin{proof}
The finiteness over $\Xmu$ and the admissibility condition of
$(\clubsuit)_\mu$ are clearly satisfied. To check the last condition of
$(\clubsuit)_\mu$, let $\Xmu\hookrightarrow \ol{\Xmu}$ be the
normalization in $\Xmu$ of $\ol{\Xlam}$ from $(\clubsuit)_\lambda$
and let $\tilde{V_\mu}=h_\mu(V_\mu)$ with
$h_\mu=(f_\mu,g_\mu):Z_\mu\to \Xmu\times\ol{M}$ ($\tilde{V_\mu}$ is
finite over $\Xmu$ by the finiteness of $V_\mu\to \Xmu$). Then
$\tilde{V_\mu}\subset \tilde{V_\lambda}\times_{\Xlam} \Xmu$ so that
the closure $\ol{\tilde{V_\mu}}$ of $\tilde{V_\mu}$ in
$\ol{\Xmu}\times\ol{M}$ is contained in
$\ol{\tilde{V_\lambda}}\times_{\ol{\Xlam}}  \ol{\Xmu}$, which is
proper over $\ol{\Xmu}$ by the assumption. Hence
$\ol{\tilde{V_\mu}}$ is also proper over $\ol{\Xmu}$, which implies
the desired condition.
\end{proof}

\begin{lemma}\label{lem3;MV}
For a commutative diagram \eqref{eq3;MV}, there is a natural induced
map
\[ \phi_* :E^\tau(Z) \to E^\tau(Z') \]
which makes $E^\tau$ a covariant functor on $\sD$.
\end{lemma}
\begin{proof}
Take $V\in E(Z)$ and let $V'=\phi(V)$. By the finiteness of $V\to
\ol{S}$, $V'$ is finite over $\ol{S}$ and closed in $Z'$. The
admissibility condition \eqref{eq3.5;MV} for $V$ implies that for
$V'$ by \cite[Lemma 1.2.1]{kmsy}. Hence $V'\in E(Z')$. Suppose $V\in
E^\tau(Z)$. To show $V'\in E^\tau(Z')$, take $\lambda\in \Lambda$
and $V_\lambda$ as in $(\clubsuit)_\lambda$. Thanks to Lemma
\ref{lem2;MV}, we may assume that the diagram \eqref{eq3;MV} is the
base change via $\ol{S}\to \Xlam$ of
\[\xymatrix{
 &\ar[ld]_{f_\lambda} Z_\lambda \ar[rd]^{g_\lambda}\ar[dd]^{\phi_\lambda}\\
\Xlam & & \ol{M}\\
&\ar[lu]^{f'_\lambda} Z'_\lambda \ar[ru]_{g'_\lambda} \\
}\] and $V'=V'_\lambda\times_{\Xlam} \ol{S}$ with
$V'_\lambda=\phi_\lambda(V_\lambda)$. Since $V_\lambda$ is finite
and surjective over a component of $X_\lambda$, so is
$V'_{\lambda}$, which implies that it is an irreducible component of
$Z'_\lambda$. The admissibility condition
\eqref{eq;admisiibilityVlam} for $V_\lambda$ implies that for
$V'_\lambda$ by \cite[Lemma 1.2.1]{kmsy}. Letting
$h_\lambda'=(f'_\lambda,g'_\lambda): Z'_\lambda\to \Xlam\times\Mb$,
we have $h_\lambda=h_\lambda' \phi_\lambda$ so that
$h_\lambda'(V'_\lambda)=h_\lambda(V_\lambda)$. Hence $V'_\lambda$
satisfies the last condition of $(\clubsuit)_\lambda$ since
$V_\lambda$ does. This implies $V'\in E^\tau(Z')$.
\end{proof}

\begin{proof}[Proof of Theorem \ref{thm1;MV} (1)]
It suffices to show the exactness of
\begin{equation}
\label{eq:ceck2s} \dots \to \Ztrfintau{U \times_M U}(S) \to
\Ztrfintau{U}(S) \to \Ztrfintau{M}(S)\to 0
\end{equation}
where $S=(\ol{S},D\times_X \ol{S})$ with $(X,D)$ and $\ol{S}$ as in
\eqref{eq;olS}. We first note that for a closed subscheme $Z\subset
\ol{S}\times \ol{U} \times_{\ol{M}} \dots \times_{\ol{M}} \ol{U}$
finite and surjective over an irreducible component of $\ol{S}$, the
image of $Z$ in $\ol{S}\times\Mb$ is finite over $\ol{S}$. From this
fact we see that \eqref{eq:ceck2s} is obtained as the inductive
limit of
\begin{equation}\label{eq:ceck3s}
 \dots \to
L^\tau(Z \times_{\Mb} (\ol{U} \times_{\Mb} \ol{U})) \to L^\tau(Z
\times_{\Mb} \ol{U}) \to L^\tau(Z) \to 0
\end{equation}
where $Z$ ranges over all closed subschemes of $\ol{S} \times \Mb$
that is finite surjective over an irreducible component of $\ol{S}$.
It suffices to show the exactness of \eqref{eq:ceck3s}.

Since $Z$ is finite over a henselian local scheme $\ol{S}$, $Z$ is a
disjoint union of henselian local schemes. Thus the Nisnevich cover
$Z \times_{\Mb} \ol{U} \to Z$ admits a section $s_0 : Z \to Z
\times_{\Mb} \ol{U}$. Define for $k \geq 0$
\[
s_k := s_0 \times_{\ol{M}} \id_{\ol{U}^k} : Z \times_{\ol{M}}
\ol{U}^k \to Z \times_{\ol{M}} \ol{U} \times_{\ol{M}} \ol{U}^k = Z
\times_{\ol{M}} \ol{U}^{k+1},
\]
where $\ol{U}^k$ is the $k$-fold fiber product of $\ol{U}$ over
$\ol{M}$. Then the maps
\[
(s_k)_* : L^\tau(Z \times_{\ol{M}} \ol{U}^k)\to L^\tau(Z
\times_{\ol{M}} \ol{U}^{k+1})
\]
give us a homotopy from the identity to zero.
\end{proof}

\subsection{Proof of Theorem \protect\ref{thm;tauexact}}

\begin{corollary}\label{tex1}\
\begin{itemize}
\item[\rm (1)] Let $G\in \ulMPST$. If $\ulaNis G=0$, then $\ulaNis \tau_! \tau^* G=0$.
\item[\rm (2)]
The base change morphism $\aNis \tau^*\Rightarrow \tau^\Nis \ulaNis$
is an isomorphism.
\item[\rm (3)] The functor $\tau^\Nis$ is exact.
\end{itemize}
\end{corollary}

\begin{proof} (1) Since $\ulaNis$, $\tau_!$ and $\tau^*$ all commute with representable colimits as left adjoints, we are reduced by Lemma \ref{lem;uaNisvanish} to $G$ of the form $\Ztr(M/U)$, which is equivalent to
\[\ulaNis\big(\Coker(\Ztrtau{U} \to \Ztrtau{M})\big)=0,\]
where the cokernel is taken in $\ulMPST$. This follows from Theorem
\ref{thm1;MV}(2).

(2) Let $F\in \ulMNST$. The base change morphism $\aNis \tau^*F\to
\tau^\Nis \ulaNis F$ is defined as the composition
\[\aNis \tau^*F\by{\aNis\tau^*(\eta_F)} \aNis \tau^*\ul{i}_\Nis \ulaNis F\simeq
\aNis i_\Nis \tau^\Nis \ulaNis F\by{\epsilon_{\tau^\Nis \ulaNis
F}}\tau^\Nis \ulaNis F\] where $\eta$ (resp. $\epsilon$) is the unit
(resp. counit) of the adjunction $(\ulaNis,\ul{i}_\Nis)$ (resp.
$(a_\Nis,i_\Nis)$). Since the second map is an isomorphism by the
full faithfulness of $i_\Nis$, it remains to show that the first one
is an isomorphism. By the full faithfulness of $\tau_\Nis$ (Lemma
\ref{lem:tau!-exact}), it suffices to show it after applying this
functor. But $\ulaNis\tau_!\simeq \tau_\Nis \aNis$ by Theorem
\ref{thm:a-nis-final}, so we are left to show that the map
\[\ulaNis\tau_! \tau^* F\by{\ulaNis\tau_! \tau^*(\eta_F)} \ulaNis\tau_! \tau^*
\ul{i}_\Nis \ulaNis F \] is an isomorphism. This follows from (1),
since $\ulaNis\tau_! \tau^*$ is exact, and $\Ker \eta_F$ and $\Coker
\eta_F$ are killed by $\ulaNis$.

Finally, (3) follows from (2), Lemma \ref{lem:lr-adjoint} (4)
and the exactness of $\tau^*$.
\end{proof}

\begin{cor} The functor $\tau^\Nis$ has a right adjoint.
\end{cor}

\begin{proof} The category $\ulMNST$ is cocomplete and has a small set of generators, as a Grothendieck category (Theorem \ref{thm:sheafification-ulMNST}.) Moreover, $\tau^\Nis$ respects all representable colimits as an exact, strongly additive functor (Lemma \ref{lem:tau!-exact} and Corollary \ref{tex1} (3)).  Thus the dual hypotheses of the ``special adjoint functor theorem'' \cite[Chapter~V, Section~8, Theorem~2]{mcl} are verified.
\end{proof}

This corollary is striking, since $\tau_s$ is not cocontinuous
\cite[Remark~5.2.1]{KaMi}.

\section{Relation with $\protect\NST$}

\subsection{$\protect\MNS, \protect\ulMNS$ and $\protect\NS$}

We consider the functors
\begin{equation}\label{eq:omega_s}
\ulomega_s : \ulMSm \to \Sm, \qquad \omega_s : \MSm \to \Sm,
\end{equation}
defined by $\ulomega_s(M)=M^\o$ and $\omega_s(M)=M^\o$,
 and the left adjoint to $\ulomega_s$, defined by $\lambda_s(X)=(X,\emptyset)$.
We have $\omega_s=\ulomega_s \tau_s$, and:

\begin{prop}[\protect{\cite[Theorem~1]{KaMi}}]\label{prop:coconti}
The functors $\omega_s : \MSm_\Nis \to \Sm_\Nis$, $\ulomega_s :
\ulMSm_\Nis \to \Sm_\Nis$ and $\lambda_s:\Sm\to \ulMSm$ are
continuous and cocontinuous.
\end{prop}

Let $\PS$ (resp. $\NS$) be the category of abelian presheaves (resp.
Nisnevich sheaves) on $\Sm$. The inclusion $i^V_{s, \Nis}:\NS \inj
\PS$ has a left adjoint $a^V_{s, \Nis}$. Let $\ulomega_s^* : \PS \to
\ulMPS$ and $\omega_s^* : \PS \to \MPS$ be the functors induced by
$\ulomega_s$ and $\omega_s$. They have left adjoints
$\ulomega_{s,!}$ and $\omega_{s,!}$.

\begin{prop}\label{prop:MNS-NS}
\leavevmode
\begin{itemize}
\item[{\rm a)}] We have $\ul\omega_{s, !}(\ulMNS) \subset \NS$ and
$\omega_{s, !}(\MNS) \subset \NS$.
\item[{\rm b)}] For $F\in \PS$, $\omega_s^*F\in \MNS$ $\iff$
$\ulomega_s^*F\in \ulMNS$ $\iff$ $F\in\NS$.
\item[{\rm c)} ] Let $\omega_s^\Nis:\NS \to \MNS$ and $\omega_{s, \Nis}:\MNS
\to \NS$ be the functors such that
\begin{align}
\label{eq:omega-a-V1s} & i_{s, \Nis} \omega_s^\Nis = \omega_s^* i_{s
,\Nis}^V, \quad i_{s, \Nis}^V \omega_{s, \Nis} = \omega_{s, !} i_{s,
\Nis},
\end{align}
which exist by b). Then $\omega_{s, \Nis}$ is left adjoint to
$\omega_s^\Nis$; both functors are exact and $\omega_s^\Nis$ is
strongly additive. We have $\omega_{s, \Nis}=a_{s, \Nis}^V
\omega_{s, !} i_{s, \Nis}$ and
\begin{equation} \label{eq:omega-a-V2s}
 \omega_{s, \Nis} a_{s, \Nis} = a_{s, \Nis}^V \omega_{s, !},
\quad \omega_s^\Nis a_{s, \Nis}^V = a_{s, \Nis} \omega_s^*.
\end{equation}
\item[{\rm d)} ] Let $\ulomega_s^\Nis:\NS \to \ulMNS$ and $\ulomega_{s,
\Nis}:\ulMNS \to  \NS$ be the functors defined in the same way as
$\omega_s^\Nis$  and $\omega_{s,\Nis}$ in c). Then the similar
statements to c) hold for $\ulomega_s^\Nis$ and $\ulomega_{s,
\Nis}$.
\end{itemize}
\end{prop}

\begin{proof}
a) Since $\lambda_s$ is left adjoint to $\ulomega_s$, we have $
\ulomega_{s, !}=\lambda_s^*$, hence the continuity of $\lambda_s$
proves the first assertion. The second one follows from Theorem
\ref{thm:MNS-tau-a}, as $\omega_{s, !}= \ulomega_{s, !} \tau_{s,
!}$.

b) If $\ulomega_s^*F \in \ulMNS$, then $\omega_s^*F = \tau_s^*
\ulomega_s^*F \in \MNS$ by Theorem \ref{thm:MNS-tau-a}. If
$\omega_s^*F \in \MNS$, then $\omega_{s, !} \omega_s^* F\iso F  \in
\NS$ by a) since $\omega_s^*$ is fully faithful. If $F \in \NS$,
then we have $\ulomega_s^*F \in \ulMNS$ since $\ulomega_s$ is
continuous.

c) The second formula of \eqref{eq:omega-a-V2s} follows from the
cocontinuity of $\omega_s$ (Proposition \ref{prop:coconti}) and
\cite[expos\'e~III, Proposition~2.3 (2)]{SGA4}. We prove the rest of the assertions
by using Lemma \ref{lem:lr-adjoint} as follows. In the situation of
Lemma \ref{lem:lr-adjoint}, set $\sC = \PS$, $\sC' = \NS$, $\sD =
\MPS$, $\sD' = \MNS$, $(i_{\sC},i_{\sD}) = (i_{s,\Nis}^V,
i_{s,\Nis})$ and $(c,c',d)=(\omega_{s}^\ast ,
\omega_{s}^{\Nis},\omega_{s,!})$.

The assumption of Lemma \ref{lem:lr-adjoint} (2) is satisfied since
$i_{s,\Nis}$ is strongly additive by Theorem
\ref{thm:sheafification-ulMNST}, $ i_{s,\Nis}^V$ is strongly
additive by the quasi-compactness of the Nisnevich topology, and
$\omega_s^\ast$ is strongly additive as a left adjoint. Hence
$\omega_{s}^{\Nis}$ is also strongly additive.

The assumption of Lemma \ref{lem:lr-adjoint} (3) is satisfied since
$i_{s,\Nis}^V, i_{s,\Nis}$ have exact left adjoints $a_{s,\Nis}^V ,
a_{s,\Nis}$, and since $\omega_{s,!}$ is exact by \cite[Proposition~2.2.1]{kmsy}. Hence the left adjoint $ \omega_{s,\Nis}$ of
$\omega_{s}^\Nis$ is exact, and we have $\omega_{s,\Nis} =
a_{s,\Nis}^V \omega_{s,!}  i_{s,\Nis}$ and $a_{s,\Nis}^V
\omega_{s,!} = \omega_{s,\Nis} a_{s,\Nis}$.

The assumption of Lemma \ref{lem:lr-adjoint} (4) is satisfied.
Indeed, the formula $a_{\sD} c = c' a_{\sC}$ coincides with the
second formula of \eqref{eq:omega-a-V2s} (which we have proven
above), and $\omega_s^\ast$ is exact as a left and right adjoint.
Hence $ \omega_{s}^\Nis$ is exact.

d) is shown by the same argument as c), by the cocontinuity of
$\ulomega_s$ (Proposition \ref{prop:coconti}) and by Lemma
\ref{lem:lr-adjoint} applied to $\sC = \PS$, $\sC' = \NS$, $\sD =
\ulMPS$, $\sD' = \ulMNS$, $(i_{\sC},i_{\sD}) = (i_{s,\Nis}^V,
\ul{i}_{s,\Nis})$ and $(c,c',d)=(\ulomega_{s}^\ast ,
\ulomega_{s}^{\Nis},\ulomega_{s,!})$.

Indeed, $\ul{i}_{s,\Nis}$ is strongly additive (Lemma \ref{lcom10}),
$\ulomega_{s}^\ast$ is exact as a left and right adjoint,
$\ul{i}_{s,\Nis}$ has an exact left adjoint $\ul{a}_{s,\Nis}$, and
$\ulomega_{s,!}$ is exact. This finishes the proof.
\end{proof}

\subsection{$\protect\MNST, \protect\ulMNST$ and $\protect\NST$}
\label{sect:NST}

Let $\PST$ be the abelian category of presheaves on $\Cor$. The
graph functor $c^V : \Sm \to \Cor$ induces an exact faithful functor
$c^{V*} : \PST \to \PS$. Let $\NST$ be the full subcategory of
$\PST$ consisting of $F \in \PST$ such that $c^{V*} F\in \NS$. The
functor $c^{V*}$ restricts to $c^{V, \Nis} : \NST \to \NS$. The
inclusion $i^V_\Nis:\NST \inj \PST$ has a left adjoint $a^V_{\Nis}$
by \cite[Theorem~3.1.4]{voetri}. By construction, it satisfies
\begin{equation}\label{eq6.2}
c^{V,\Nis} a_\Nis^V = a^V_{s,\Nis}c^{V*}.
\end{equation}

We consider the functors
\begin{equation}
\ulomega : \ulMCor \to \Cor, \qquad \omega : \MCor \to \Cor,
\end{equation}
defined in the same way as \eqref{eq:omega_s}. They induce
$\ulomega^* : \PST \to \ulMPST$ and $\omega^* : \PST \to \MPST$,
which have left adjoints $\ulomega_{!}$ and $\omega_{!}$. One has
the obvious identifications
\begin{equation}\label{eq6.1}
\ul{c}^\ast \ulomega^\ast  = \ulomega_s^\ast c^{V\ast }, \quad
c^\ast \omega^\ast  = \omega_s^\ast c^{V \ast}.
\end{equation}
One also sees from \cite[(2.2.1)]{kmsy} (and its analogues for
$\omega_s, \ulomega_s$) that
\begin{equation}\label{eq:cv-omega}
c^{V*} \ulomega_! = \ulomega_{s, !} \ul{c}^*, \quad c^{V*} \omega_!
= \omega_{s, !} c^*.
\end{equation}

\begin{prop}\label{prop:MNST-NST}
\leavevmode
\begin{itemize}
\item[{\rm a)}]  We have $\ul\omega_!(\ulMNST) \subset \NST$ and
$\omega_!(\MNST) \subset \NST$.
\item[{\rm b)}] For $F\in \PST$, $\omega^*F\in \MNST$ $\iff$
$\ulomega^*F\in \ulMNST$ $\iff$ $F\in\NST$.
\item[{\rm c)}] Let $\omega^\Nis:\NST \to \MNST$ and $\omega_\Nis:\MNST \to
\NST$ be the functors such that
\begin{align}
\label{eq:omega-a-V1} & i_\Nis \omega^\Nis = \omega^* i_\Nis^V,
\quad i_\Nis^V \omega_\Nis = \omega_! i_\Nis,
\end{align}
where the second equality shows that $\omega_\Nis = a^V_\Nis
\omega_! i_\Nis$ by $a^V_\Nis i^V_\Nis = \id$. Then $\omega_\Nis$ is
left adjoint to $\omega^\Nis$, and  we have
\begin{align}
\label{eq:omega-a-V2} & \omega_\Nis a_\Nis = a_\Nis^V \omega_!,
\quad \omega^\Nis a_\Nis^V = a_\Nis \omega^*,
\\
\label{eq:cvo-nis} & c^\Nis \omega^\Nis  = \omega_s^\Nis c^{V,
\Nis}.
\end{align}
Moreover, the functors $\omega_\Nis$ and $\omega^\Nis$ are both
exact, $\omega^\Nis$ is  fully faithful,
strongly additive and preserves injectives.
\item[{\rm d)}] Let $\ulomega^\Nis:\NST \to \ulMNST$ and
$\ulomega_\Nis:\ulMNST \to  \NST$ be the functors defined in the
same way as $\omega^\Nis$ and $\omega_\Nis$ in c). Then the similar
statements as c) holds for $\ulomega^\Nis$ and $\ulomega_\Nis$.
\end{itemize}
\end{prop}

\begin{proof}
a) First we prove $\ulomega_! (\ulMNST) \subset \NST$. Since
$\ul{c}^*(\ulMNST) \subset \ulMNS$ (see Lemma-Definition
\ref{lem:mnst-condition}), we have
\[
c^{V\ast} \ulomega_! (\ulMNST) = \ulomega_{s,!} \ul{c}^\ast
(\ulMNST) \subset \ulomega_{s,!} (\ulMNS) \subset \NS ,
\]
where the first equality follows from the first equality of
\eqref{eq:cv-omega}, and the last inclusion follows from Proposition
\ref{prop:MNS-NS} a). Then by definition of $\NST$ we have
$\ulomega_! (\ulMNST) \subset \NST$, as desired. The inclusion
$\omega_! (\MNST) \subset \NST$ follows from this and
Lemma-Definition \ref{lem:mnst-condition2}.

b) follows from Proposition \ref{prop:MNS-NS} b),
\eqref{eq:cv-omega} and the definitions of $\NST$, $\ulMNST$ and
$\MNST$.

c) The full faithfulness of $\omega^\Nis$ follows from that of
$\omega^*$ \cite[Proposition~2.2.1]{kmsy}, and \eqref{eq:cvo-nis} follows
from \eqref{eq6.1}. The strategy of the rest of the proof is similar
to that of c) in Proposition \ref{prop:MNS-NS}. In the situation of
Lemma \ref{lem:lr-adjoint}, set $\sC = \PST$, $\sC' = \NST$, $\sD =
\MPST$, $\sD' = \MNST$, $(i_{\sC},i_{\sD}) = (i_{\Nis}^V, i_{\Nis})$
and $(c,c',d)=(\omega^\ast , \omega^{\Nis},\omega_{!})$.

The assumptions of Lemma \ref{lem:lr-adjoint} (2) and (3) are
satisfied, since $i_{\Nis}$ is strongly additive by Theorem
\ref{thm:a-nis-final} (1), $i_{\Nis}^V$ is strongly additive by the
quasi-compactness of Nisnevich cohomology, $\omega^\ast$ is strongly
additive as a left and right adjoint, $i_{\Nis}^V$, $i_{\Nis}$ have
exact left adjoints $a_{\Nis}^V, a_{\Nis}$ by \cite{mvw} and
\ref{thm:a-nis-final} (2), and $\omega_!$ is exact by \cite[Proposition~2.2.1]{kmsy}. Therefore, the assertions follow, except for the
second identity of \eqref{eq:omega-a-V2} and the exactness of
$\omega^\Nis$. (Note that the exactness of $\omega_\Nis$ implies
that $\omega^\Nis$ preserves injectives.)

We can prove the second identity of \eqref{eq:omega-a-V2} as
follows. Its first identity yields a base change morphism
\[a_\Nis \omega^*\Rightarrow \omega^\Nis a_\Nis^V. \]

Let $F\in \PST$. We want to show that the morphism $a_\Nis
\omega^*F\to \omega^\Nis a_\Nis^V F$ is an isomorphism. Since
$c^\Nis$ is conservative as $\MSm$ and $\MCor$ have the same
objects, it suffices to show that the induced morphism $c^\Nis
a_\Nis \omega^\ast F \to c^\Nis \omega^\Nis a_\Nis^V F$ is an
isomorphism. Since
\begin{align*}
c^\Nis a_\Nis \omega^\ast &\overset{\eqref{eq:Anis-c-tau}}{=} a_{s,\Nis} c^\ast \omega^\ast \overset{\eqref{eq6.1}}{=}  a_{s,\Nis} \omega_s^\ast c^{V\ast}, \\
c^\Nis \omega^\Nis a_\Nis^V &\overset{\eqref{eq:cvo-nis}}{=}
\omega_s^\Nis c^{V,\Nis} a_\Nis^V \overset{\eqref{eq6.2}}{=}
\omega_s^\Nis a_{s,\Nis}^V c^{V\ast} ,
\end{align*}
the above morphism is rewritten as
\[
a_{s,\Nis} \omega_s^\ast c^{V\ast} F \to \omega_s^\Nis a_{s,\Nis}^V
c^{V\ast} F,
\]
which is an isomorphism by Proposition \ref{prop:MNS-NS} c).

Now, the formula we have proven and the exactness of
$\omega^\ast$ as a left and right adjoint show that the assumption
of Lemma \ref{lem:lr-adjoint} (4) is satisfied. Hence $\omega^\Nis$
is exact, as desired. This finishes the proof of c).

d) The proof is completely parallel to that of c). To see this, it
suffices to observe the following: $\ul{i}_\Nis$ is strongly
additive and has an exact left adjoint by \cite[Lemma~4.5.3, Theorem~4.5.5]{kmsy}, $\ulomega^\ast$ is strongly additive as a left and
right adjoint, $\ulomega_!$ is exact by \cite[Proposition~2.3.1]{kmsy},
$\ul{c}^\Nis$ is conservative since $\ulMSm$ and $\ulMCor$ have the
same objects, the base change morphism $\ul{a}_{s,\Nis}
\ulomega_s^\ast \Rightarrow \ulomega_s^\Nis \ul{a}_{s,\Nis}^V$ is an
isomorphism by Proposition \ref{prop:MNS-NS} d), and we have the
following identifications:
\begin{align*}
\ul{c}^\Nis \ul{a}_\Nis \ulomega^\ast &\overset{\eqref{eq:a-c2}}{=} \ul{a}_{s,\Nis} \ul{c}^\ast \ulomega^\ast \overset{\eqref{eq6.1}}{=}  \ul{a}_{s,\Nis} \ulomega_s^\ast c^{V\ast}, \\
\ul{c}^\Nis \ulomega^\Nis a_\Nis^V &\overset{\eqref{eq:cvo-nis}}{=}
\ulomega_s^\Nis c^{V,\Nis} a_\Nis^V \overset{\eqref{eq6.2}}{=}
\ulomega_s^\Nis a_{s,\Nis}^V c^{V\ast} .
\end{align*}
This finishes the proof.
\end{proof}

\subsection{Relation between cohomologies}

We now prove Theorem \ref{thm;MNSTNST} (2) from the introduction.

\begin{lemma}\label{lem:inj-flasque-NST}
Let $I \in \MNST$ be an injective object. Then $\omega_\Nis I \in
\NST$ is flasque.
\end{lemma}
\begin{proof}
Let $U \to X$ in $\Sm$ be an open dense immersion. We need to show
the surjectivity of
\[ \omega_\Nis I(X)
= \colim_{M \in \MSm(X)} I(M) \to
  \omega_\Nis I(U)
= \colim_{N \in \MSm(U)} I(U).
\]
We fix $M \in \MSm(X)$ and will show that the composition $I(M) \to
\omega_\Nis I(X) \to \omega_\Nis I(U)$ is already surjective. This
follows from the functoriality of $\omega^!$, but for clarity we
give an explicit argument. Take any $N \in \MSm(U)$. Let $\ol{N}'$
be the closure of the image of $U \to \ol{N} \times \ol{M}$, and let
$\ol{L}$ be the blow-up of $\ol{N}'$ along $(\ol{N}' \setminus
U)_\red$. Set $L:=(\ol{L}, L^\infty) \in \MCor$ where $L^\infty$ is
the pull-back of $N^\infty$ along the composition $\ol{L} \to
\ol{N}' \to \ol{N}$. We have a commutative diagram in $\ulMCor$
\[
\xymatrix{ & (U, \emptyset) \ar[r] \ar[d] \ar[ld] & (X, \emptyset)
\ar[d]
\\
N & L^{(n)} \ar[r] \ar[l] & M }
\]
for sufficiently large $n$ (see \cite[Definition~1.4.1]{kmsy}). Since $U
\to X$ is an open dense immersion, $\Cor(V, U) \to \Cor(V, X)$ is
injective for any $V \in \Sm$, which in turn implies the injectivity
of the morphism $\Z_\tr(L^{(n)}) \to \Z_\tr(M)$ in $\MNST$. Since $I$ is an
injective object, we conclude that $I(M) \to I(L^{(n)})$ is
surjective. This proves the lemma, as the canonical map $I(N) \to
\omega_\Nis I(U)$ factors through $I(L^{(n)})$.
\end{proof}

\begin{thm}
\leavevmode
\label{prop:MNST-NST-coh}
\begin{itemize}
\item[\rm (1)] For any $M \in \MCor, ~G \in \NST$ and $p \ge 0$,
we have a canonical isomorphism
\[
\Ext_{\MNST}^p(\Z_\tr(M), \omega^\Nis G) \cong H^p_\Nis(M^\o, G).
\]
\item[\rm (2)]
For any $X \in \Sm, ~F \in \MNST$ and $p \ge 0$, we have a canonical
isomorphism
\[
H^p_\Nis(X, \omega_\Nis F) \cong \colim_{M \in \MSm(X)}
H^p_\Nis(\ol{M}, F_M).
\]
\end{itemize}
\end{thm}

\begin{proof}
It follows from Proposition \ref{prop:MNST-NST} and Theorem
\cite[Theorem~A.9.1]{kmsy} that
\[ (R^p i_\Nis)\omega^\Nis = R^p(i_\Nis \omega^\Nis)
= R^p( \omega^* i_\Nis^V) = \omega^* R^p i_\Nis^V.
\]
By
Proposition~\ref{pA.2} and the projectivity of $\Z_\tr(M)$ and
$\Z_\tr^V(M^\o)$, we get
\begin{align*}
&\Ext_{\MNST}^p(\Z_\tr(M), \omega^\Nis G) \cong \MPST(\Z_\tr(M),
(R^p i_\Nis)\omega^\Nis G)
\\
&\cong \MPST(\Z_\tr(M), \omega^* R^p i_\Nis^V G) \cong \PST(\omega_!
\Z_\tr(M), R^p i_\Nis^V G)
\\
&\cong \PST(\Z_\tr^V(M^\o), R^p i_\Nis^V G) \cong H^p_\Nis(M^\o, G),
\end{align*}
whence (1).

Given $X \in \Sm$, we define functors $\Gamma_X^\to : \MNST \to \Ab$
and $\Gamma_X^V : \NST \to \Ab$ by
\[ \Gamma_X^\to(F)
=\colim_{M \in \MSm(X)} F(M), \qquad \Gamma_X^V(G)=G(X).
\]
We have $\Gamma_X^\to=\Gamma_X^V \omega_\Nis$. By \cite[Theorem~A.9.1]{kmsy} and Lemma \ref{lem:inj-flasque-NST}, we get $(R^p
\Gamma_X^V)\omega_\Nis=R^p \Gamma_X^\to$ for any $p \ge 0$, since
$\omega_\Nis$ is exact. Taking an injective resolution $F \to
I^\bullet$ in $\MNST$, we proceed
\begin{align*}
R^p \Gamma_X^\to(F) &\cong H^p(\Gamma_X^\to I^\bullet) \cong
H^p(\colim_{M \in \MSm(X)} I^\bullet(M))
\\
&\cong \colim_{M \in \MSm(X)} H^p(I^\bullet(M)) \qquad
\text{\cite[Lemma 1.7.4]{kmsy}}
\\
&\cong \colim_{M \in \MSm(X)} \colim_{N \in \Sigma^\fin \downarrow
M} H^p_\Nis(\ol{N}, F_N) \qquad \text{(Theorem \ref{thm:coh-MNST})}
\\
&\cong \colim_{M \in \MSm(X)} H^p_\Nis(\ol{M}, F_M).
\end{align*}
Here the last isomorphism holds since any $N \in \Sigma^\fin
\downarrow M$ gives rise to an object $N \in \MSm(X)$. This
concludes the proof of (2).
\end{proof}

\section{Passage to derived categories}

In this section, we extend the previous results to derived
categories. The main result is an extension of Proposition
\ref{c3.1v} and Theorem \ref{thm:coh-MNST} to unbounded complexes
(Proposition \ref{lem;cohMsigmaST} (2) and Theorem \ref{tmain} (3)).

\subsection{Compactness}

\begin{prop}\label{p7.1} Let $M\in \ulMSm$. Then the following objects are compact:
\begin{itemize}
\item[\rm (1)] $\Z_\tr(M)[0]$ in $D(\ulMNST)$, and in $D(\MNST)$ if $M$ is proper.
\item[\rm (2)] $\Z(M)[0]$ in $D(\ulMNS^\fin)$ and $D(\ulMNS)$.
\end{itemize}
\end{prop}

\begin{proof} (1) follows from Proposition \ref{c3.1v} (resp. from Theorem \ref{thm:coh-MNST}) and the known commutation of Nisnevich cohomology with filtered colimits of sheaves, via  hypercohomology spectral sequences which are convergent and Corollary \ref{c4.1} (resp. \ref{c5.1}). (2) is seen similarly, using \cite[(3.6.1) and Proposition~4.4.2]{kmsy}.
\end{proof}

\subsection{Strong additivity}

\begin{thm}\label{t7.1}\
\begin{itemize}
\item[\rm (1)] $($\textit{cf.} \cite[Proposition~4.3.3 (2)]{kmsy}$)$ The functor $R\ul{b}_s^\Nis:D(\ulMNS)\to D(\ulMNS^\fin)$ is right adjoint to $D(\ul{b}_{s,\Nis})$ and is strongly additive. The counit map $D(\ul{b}_{s,\Nis})R\ul{b}_s^\Nis\allowbreak\to \id$ is an isomorphism.
\item[\rm (2)] The functor $R(\ul{b}_s^\Nis \circ \ul{c}^\Nis)$ is strongly additive.
\end{itemize}
\end{thm}

\begin{proof} (1) The adjunction statement follows from Proposition \ref{pA.3}. The second assertion implies the third by Lemma \ref{layoub}. It remains to prove the strong additivity of $R\ul{b}_s^\Nis$. For this, we check that the hypothesis of Proposition \ref{pA.1} c) are verified: we take the $\Z(M)$, $M\in \MSm^\fin$, as a set of generators. We have $\ul{b}_{s,\Nis} \Z(M)=\Z(M)$. Their compactness follows from Proposition \ref{p7.1} (2).

(2) Same argument as (1), using Proposition \ref{p7.1} (1) as well.
\end{proof}

\subsection{From $D(\protect \ulMNS^\fin)$ to $D(\protect \ulMNS)$} We first extend \cite[Notation 4.4.1]{kmsy} from sheaves to complexes of sheaves:

\begin{nota}\label{n3.7b} Let $N \in \ulMP^\fin$ and  let $K$ be a complex on $\ulMNS^\fin$.
We write $K_N$ for the complex of sheaves on $(\ol{N})_\Nis$ deduced
from $K$ via the isomorphism of sites from \cite[Lemma 3.1.3]{kmsy}.
If $K$ is a complex on $\ulMNS$, we write $K_N$ for $(\ul{b}_s^\Nis
K)_N$.
\end{nota}

The following extends \cite[Proposition~4.4.2]{kmsy} to unbounded
complexes of sheaves.

\begin{prop}\label{lem;cohMsigmaS} Let  $M\in \ulMSm$.
For $K\in D(\ulMNS)$, we have natural isomorphisms
\begin{equation*}
\Hom_{D(\ulMNS)}(\Z_\tr(M),K[i])\simeq \colim_{N\in
\ul{\Sigma}^{\fin}\downarrow M} \bH_\Nis^i(\ol{N},K_N) \simeq
\colim_{N\in \ul{\Sigma}^{\fin}\downarrow M}
\bH_\Nis^i(\ol{N},(R\ul{b}_s^\Nis K)_N), \quad i\in\Z.
\end{equation*}
\end{prop}

\begin{proof} Define as in \emph{loc.~cit.}  functors
$\Gamma_M^\downarrow : \ulMNS^\fin \to \Ab$ and $\ul{\Gamma}_M :
\ulMNS \to \Ab$ by
\[ \Gamma_M^\downarrow(G)
=\colim_{N \in \ul{\Sigma}^\fin \downarrow M} G(N), \quad
\ul{\Gamma}_M(F)=F(M).
\]
We have $\Gamma_M^\downarrow=\ul{\Gamma}_M \ul{b}_{s, \Nis}$ and we
shall show that the natural transformation \eqref{eqA.1}
\begin{equation}\label{eq7.1}
R\Gamma_M^\downarrow\Rightarrow R\ul{\Gamma}_M \circ
D(\ul{b}_{s,\Nis})
\end{equation}
is invertible. For this, we apply Lemma \ref{layoub}. Its first
condition is given by \cite[Lemma 4.4.3]{kmsy}, which says that
$\ul{b}_{s,\Nis}$ sends injectives to flabbys; by \cite[Theorem~A.9.1
and Lemma A.9.3]{kmsy}, this  already yields isomorphisms
\begin{equation}\label{eqII7.1}
R^p\Gamma_M^\downarrow\iso R^p\ul{\Gamma}_M \circ \ul{b}_{s,\Nis},
\quad p\ge 0.
\end{equation}

We are now left to show the strong additivity of the three functors.
For $D(\ul{b}_{s,\Nis})$, this follows from  the strong additivity
of $\ul{b}_{s,\Nis}$ as a left adjoint, and  Proposition \ref{pA.1}
a).

For $R\Gamma_M^\downarrow$ and $R\ul{\Gamma}_M$, we check that the
conditions of Proposition \ref{pA.1} b) are verified. For
$R\ul{\Gamma}_M$,  Condition (ii) follows from the vanishing
statement in (1) (use the compact projective generator $\Z$ of
$\Ab$), and Condition (i) follows similarly from the known
commutation of Nisnevich cohomology with filtered colimits of
sheaves. The case of $R\Gamma_M^\downarrow$ is reduced to this one
by \eqref{eqII7.1}.

Thus we get an isomorphism
\[\Hom_{D(\ulMNS)}(\Z_\tr(M),\ul{b}_{s,\Nis} L[i])\simeq \colim_{N\in \ul{\Sigma}^{\fin}\downarrow M} \bH_\Nis^i(\ol{N},L_N), \quad i\in\Z\]
for any complex $L$ on $\ulMNS^\fin$. Setting $L=\ul{b}_s^\Nis K$,
we get the first isomorphism thanks to the isomorphism
$\ul{b}_{s,\Nis}\ul{b}_s^\Nis\iso \id$ of \cite[Proposition~4.3.3
(2)]{kmsy}. Composing \eqref{eq7.1} with $R\ul{b}_s^\Nis$ and using
Theorem \ref{t7.1} (1), we get an isomorphism
\begin{equation}\label{eq7.2}
R\Gamma_M^\downarrow\circ R\ul{b}_s^\Nis\iso R\ul{\Gamma}_M
\end{equation}
which yields the second isomorphism of Proposition
\ref{lem;cohMsigmaS}.
\end{proof}

\subsection{From $D(\protect \ulMNS)$ to $D(\protect \ulMNST)$}

\begin{nota}\label{n3.7c} Let $N \in \ulMP^\fin$ and  let $K$ be a complex on $\ulMNST^\fin$.
We write $K_N$ for $(\ul{b}_s^\Nis\ul{c}^\Nis K)_N$.
\end{nota}

The following extends Proposition \ref{c3.1v} to unbounded complexes
of sheaves.

\begin{prop}\label{lem;cohMsigmaST} Let  $M\in \ulMCor$. For  $K\in D(\ulMNST)$,
we have a natural isomorphism
\begin{equation*}
\Hom_{D(\ulMNST)}(\Z_\tr(M),K[i])\simeq \colim_{N\in
\ul{\Sigma}^{\fin}\downarrow M} \bH_\Nis^i(\ol{N},K_N) \simeq
\colim_{N\in \ul{\Sigma}^{\fin}\downarrow M}
\bH_\Nis^i(\ol{N},(R\ul{b}_s^\Nis D(\ul{c}^\Nis) K)_N), \quad
i\in\Z.
\end{equation*}
\end{prop}

\begin{proof} Define as before  functors
$\ul{\Gamma}_M : \ulMNS \to \Ab$ and $\ul{\Gamma}_M^T : \ulMNST \to
\Ab$ by
\[\ul{\Gamma}_M(F)=F(M), \quad \ul{\Gamma}_M^T(G)
= G(M).
\]

We have $\ul{\Gamma}_M^T=\ul{\Gamma}_M \circ \ul{c}^\Nis$ thanks to
the adjunction $(\ul{c}_\Nis,\ul{c}^\Nis)$ of Theorem
\ref{thm:sheafification-ulMNST}. Moreover, $\ul{c}^\Nis$ is exact by
this theorem. Let us show as before that the natural transformation
\eqref{eqA.1}
\begin{equation}\label{eq7.3}
R\ul{\Gamma}_M^T\Rightarrow R\ul{\Gamma}_M \circ D(\ul{c}^\Nis)
\end{equation}
is invertible. We copy the argument of the previous subsection,
using  Lemma \ref{layoub}.

Its first condition is given by \cite[Lemma 4.6.1]{kmsy}, which says
that $\ul{c}^\Nis$ sends injectives to flabbys; by \cite[Theorem~A.9.1
and Lemma A.9.3]{kmsy}, this  already yields isomorphisms
\begin{equation}\label{eqII7.1T}
R^p\ul{\Gamma}_M^T\iso R^p\ul{\Gamma}_M \circ \ul{c}^\Nis, \quad
p\ge 0.
\end{equation}

We are now left to show the strong additivity of the three functors.
For $D(\ul{c}^\Nis)$, this follows from  the strong additivity of
$\ul{c}^\Nis$ (Theorem \ref{thm:sheafification-ulMNST}) and
Proposition \ref{pA.1} a). For the two other functors, we need to
check the conditions of Proposition \ref{pA.1} b); this was done for
$R\ul{\Gamma}_M$ in the previous section, and the case of
$R\ul{\Gamma}_M^T$ is reduced to this one by \eqref{eqII7.1T}.

Thus we get an isomorphism
\[\Hom_{D(\ulMNST)}(\Z_\tr(M), L[i])\iso \Hom_{D(\ulMNS)}(\Z_\tr(M),\ul{c}^\Nis L[i])\quad i\in\Z\]
for any complex $L$ on $\ulMNST$, and we get the first isomorphism
by using Proposition \ref{lem;cohMsigmaS}. For the second one,
composing \eqref{eq7.2} with \eqref{eq7.3}, we obtain an isomorphism
\begin{equation}\label{eq7.5}
R\ul{\Gamma}_M^T\simeq R\Gamma_M^\downarrow\circ R\ul{b}_s^\Nis
\circ D(\ul{c}^\Nis)
\end{equation}
which yields the second isomorphism of Proposition
\ref{lem;cohMsigmaST}.
\end{proof}

\subsection{From $D(\protect \ulMNST)$ to $D(\protect \MNST)$}

\begin{thm}\label{tmain} \
\begin{itemize}
\item[\rm (1)] The functor $D(\tau_\Nis):D(\MNST)\to D(\ulMNST)$ is fully faithful.
\item[\rm (2)] One has isomorphisms
\begin{eqnarray*}
\Hom_{D(\MNST)}(\Z_\tr(M),K[i]) & \simeq &  \colim_{N\in \ul{\Sigma}^{\fin}\downarrow M} \bH_\Nis^i(\ol{N},K_N)\\
 & \simeq & \colim_{N\in \ul{\Sigma}^{\fin}\downarrow M} \bH_\Nis^i(\ol{N},(R\ul{b}_s^\Nis D(\ul{c}^\Nis \tau_\Nis))K_N), \quad i\in\Z
\end{eqnarray*}
for any complex $K$ on $\MNST$, where
$K_N:=(\ul{b}_s^\Nis\ul{c}^\Nis\tau^\Nis K)_N$.
\end{itemize}
\end{thm}

\begin{proof}
Recall that $\tau_\Nis$ and $\tau^\Nis$ are both exact and strongly
additive: see Lemma \ref{lem:tau!-exact} (2) and Theorem
\ref{thm;tauexact} ($\tau_\Nis$ is strongly additive as a left
adjoint). Then we get from Proposition \ref{lA.2} an adjunction
$(D(\tau_\Nis),D(\tau^\Nis))$ and from Lemma \ref{layoub} an
isomorphism
\[D(\tau^\Nis)D(\tau_\Nis) \osi D(\tau^\Nis\tau_\Nis)\iso D(\id) = \id \]
which shows the full faithfulness of $D(\tau_\Nis)$. This shows (1).
(2) now follows from (1) and Proposition \ref{lem;cohMsigmaST} (2).
\end{proof}

\begin{rk} One can show that the essential image of $D(\tau_\Nis)$ is
\[D_{\MNST}(\ulMNST) :=  \{K\in D(\ulMNST)\mid H^i(K)\in \tau_\Nis(\MNST) \ \forall i\in \Z \}.\]
Since the proof involves delicate and lengthy arguments relying on
the notion of left-completeness, we skip it (see \cite{motmod}).
\end{rk}

\subsection{$D(\protect \MNST)$, $D(\protect \ulMNST)$ and $D(\protect \NST)$}
We leave it to the reader to produce an unbounded version of Theorem
\ref{prop:MNST-NST-coh}.

\section*{Appendix: Categorical toolbox, II}
\addcontentsline{toc}{section}{Appendix: Categorical toolbox,
II}\refstepcounter{section}\renewcommand{\thesection}{A}
\label{sect:app}

\subsection{A spectral sequence} The following convenient proposition is used several times in the paper.

\begin{prop}\label{pA.2} Let $a:\sB\leftrightarrows\sA:i$ be a pair of adjoint functors between abelian categories $(a$ is left adjoint to $i)$. Suppose that $\sA$ has enough injectives and that $a$ is exact. Then, for any $(A,B)\in \sA\times \sB$, there is a convergent spectral sequence
\[\Ext_\sB^p(B,R^qi A)\Rightarrow \Ext_\sA^{p+q}(aB,A).\]
If $B$ is projective, this spectral sequence collapses to
isomorphisms
\begin{equation}\label{eq:Ext-projB}
\sB(B,R^qi A)\simeq \Ext_\sA^{q}(aB,A).
\end{equation}
\end{prop}

\begin{proof} Fix $B$. By adjunction, the composition of functors
\[\sA\by{i}\sB\by{\sB(B,-)}\Ab\]
is isomorphic to $\sA(aB,-)$. We then get the spectral sequence from
\cite[Theorem~A.9.1, Example~A.9.2]{kmsy}.
The last fact is obvious.
\end{proof}

\subsection{Unbounded derived categories}

\begin{thm}[\protect{\cite[Theorem~14.3.1]{ks}}]\label{tA.1} Let $\sA$ be a Grothendieck category.
\begin{itemize}
\item[{\rm a)}]  Let $K(\sA)$ be the unbounded homotopy category of $\sA$. The localisation functor $\lambda_\sA:K(\sA)\to D(\sA)$ has a right adjoint $\rho_\sA$, whose essential image is $($by definition$)$ the full subcategory of \emph{homotopically injective} complexes.\footnote{This is the same notion as Spaltenstein's $K$-injective \cite{spaltenstein}.}
\item[{\rm b)}] Let $F:K(\sA)\to \sT$ be a triangulated functor. Then $F$ has a $($universal$)$ right Kan extension $RF$ relative to $\lambda_\sA$, given by $RF=F\circ \rho_\sA$. In particular, any left exact functor $F:\sA\to \sB$, where $\sB$ is another abelian category, has a total right derived functor $RF:D(\sA)\to D(\sB)$ given by $RF=\lambda_\sB\circ K(F)\circ \rho_\sA$.
\item[{\rm c)}] The restriction of $RF$ to $D^+(\sA)$ is the total derived
functor $R^+F$ $($\emph{cf.} \cite[\S 2, Remark~1.6]{CD}$)$.
\end{itemize}
\end{thm}

Let us justify c), which is not in \cite{ks}: the point is that
$\rho_\sA$ carries $D^+(\sA)$ into $K^+(\sA)$ \cite[Theorem~13.3.7]{ks}.

\begin{defn}\label{dA.1} Let $F:\sA\to \sB$ be a left exact functor between Grothendieck categories. An object $C\in K(\sA)$ is \emph{$F$-acyclic} if the morphism
\[\lambda_\sB K(F)C\to RF\lambda_\sA C = \lambda_\sB K(F) \rho_\sA \lambda_\sA C\]
given by the unit map of the adjunction $(\lambda_\sA,\rho_\sA)$ is
an isomorphism.
\end{defn}

\begin{ex}\label{exA.1} If $F$ is exact, every object of $K(\sA)$ is $F$-acyclic.
\end{ex}

Let $\sA\by{F}\sB\by{G}\sC$ be a chain of left exact functors
between Gro\-then\-dieck categories. The unit map of the adjunction
$(\lambda_\sB,\rho_\sB)$ yields a natural transformation
\begin{equation}\label{eqA.1}
R(GF)\Rightarrow RG \circ RF.
\end{equation}

The following lemma is made tautological by Definition \ref{dA.1}:

\begin{lemma}\label{lA.2} \eqref{eqA.1} is a natural isomorphism if and only if $F$ carries homotopically injectives to $G$-acyclics. In particular, \eqref{eqA.1} is a natural isomorphism provided $G$ is exact (see Example \ref{exA.1}).\qed
\end{lemma}

 Here is a first application:

\begin{prop}\label{pA.3} Assume $\sC=\sA$, $F$ right adjoint to $G$, and $G$ exact. Then $RF$ is right adjoint to $RG=D(G)$.
\end{prop}

\begin{proof} This is a special case of \cite[Theorem~14.4.5]{ks}.
\end{proof}

We come back to the general situation. Suppose that $F$ carries
injectives of $\sA$ to $G$-acyclics. Then \cite[Theorem~A.9.1]{kmsy}
implies that \eqref{eqA.1} is an isomorphism when restricted to
$D^+(\sA)$ (\cite[\S 2, Proposition~3.1]{CD}, \cite[Theorem~13.3.7 and Proposition~13.3.13]{ks}). This is not true on $D(\sA)$ in general, as pointed
out by Ayoub and Riou:

\begin{ex} \label{exA.2}Let  $\sB=\Mod\Z[\Z/2]$, $\sA=\sB^\N$ and $\sC=\Ab$; let $F=\bigoplus_\N$ and $G=H^0(\Z/2,-)$. The above hypotheses are verified: since $\Z[\Z/2]$ is Noetherian, a direct sum of injectives is injective. Let $M= (\Z/2[n])_{n\in \N}\in D(\sA)$. We claim that the map
\begin{equation}\label{eqA.2}
R(GF)(M)\to RG\circ RF(M)
\end{equation}
is not an isomorphism. Indeed, $GF = F'G'$ where $G':\sA\to\Ab^\N$
is $H^0(\Z/2,-)$ and $F':\Ab^\N\to \Ab$ is $\bigoplus_\N$. Let
$C=RG(\Z/2)$, so that $H^q(C)=\Z/2$ for $q\ge 0$ and $H^q(C)=0$ for
$q<0$. Then, by Lemma \ref{lA.2}:
\[R(GF)(M) = R(F'G')(M)=RF'\circ RG'(M) = \bigoplus_{n\in \N} C[n].\]

On the other hand,
\[RG\circ RF(M) = RG(\bigoplus_{n\in \N} \Z/2[n]).\]

But, in $D(\sB)$, we have $\bigoplus_{n\in \N} \Z/2[n]\iso
\prod_{n\in \N} \Z/2[n]$, and $RG$ commutes with products as a right
adjoint. Hence
\[RG\circ RF(M)=\prod_{n\in \N} C[n].\]

For $q\in \Z$, we have $H^q(\bigoplus\limits_{n\in \N}
C[n])=\bigoplus\limits_{q+n\ge 0} \Z/2$ and $H^q(\prod\limits_{n\in
\N} C[n])=\prod\limits_{q+n\ge 0} \Z/2$.
\end{ex}

However, we have the following lemma of Ayoub:

\begin{lemma}\label{layoub} Suppose that $F$ carries injectives to $G$-acyclics and that $RF,RG$ and $R(GF)$ are strongly additive. Then \eqref{eqA.1} is an isomorphism.
\end{lemma}
(In example \ref{exA.2}, $RG$ is not strongly additive.)

\begin{proof}[Proof (Ayoub).] Let $M\in D(\sA)$. We have to show that \eqref{eqA.2} is an isomorphism. Viewing $M$ as an object of $K(\sA)$, we have an isomorphism
\[\hocolim_n \sigma_{\ge n} M\iso M\]
where $\sigma_{\ge n}$ is the stupid truncation. This isomorphism
still holds in $D(\sA)$, because $\lambda_\sA$ is strongly additive.
By the hypothesis, this reduces us to the case where $M\in
D^+(\sA)$, and therefore to Grothendieck's theorem (\emph{cf.} Theorem
\ref{tA.1} c)).
\end{proof}

Let $F:\sA\to \sB$ be a left exact functor between Grothendieck
categories. In view of Lemma \ref{layoub}, we need a practical
sufficient condition to ensure that $RF$ is strongly additive. The
following ones are adapted to the context of this paper:

\begin{prop} \label{pA.1}  
\leavevmode
\begin{itemize}
\item[{\rm a)}] If $F$ is strongly additive and exact, $RF=D(F)$ is strongly additive.
\item[{\rm b)}] Suppose that
\begin{thlist}
\item For any $p\ge 0$, $R^p F$ is strongly additive.
\item There exists a set $\sE$ of compact projective  generators of $\sB$ such that, for any $E\in \sE$, there is an integer $cd_F(E)$ such that
\[\sB(E,R^p F(A)) = 0 \text{ for } p> cd_F(E) \text{ and for all } A\in \sA.\]
\end{thlist}
Then $RF$ is strongly additive.
\item[{\rm c)}] Suppose that $RF$ admits a left adjoint $G$ which sends a
set $(E_\alpha)$ of compact generators of $D(\sB)$ to compact
objects of $D(\sA)$. Then $RF$ is strongly additive.
\end{itemize}	
\end{prop}

\begin{proof} a) The strong additivity of $F$ easily implies that of $K(F)$, which in turn implies that of $D(F)$ since $\lambda_\sB$ is strongly additive as a left adjoint.

b) Let $(C_i)_{i\in I}\in D(\sA)^I$. We must show that the map
\begin{equation}\label{eqA.3}
\bigoplus_{i\in I} RF(C_i)\to RF(\bigoplus_{i\in I} C_i)
\end{equation}
is an isomorphism. Since the $E[n]$, $E\in \sE$, $n\in \Z$, are a
set of generators of $D(\sB)$, it suffices to check this after
applying $D(\sB)(E[n],-)$ for all $(E,n)$. Since $E$ is projective,
we have an isomorphism
\[D(\sB)(E[n],D)\simeq \sB(E,H^n(D))\]
for any $D\in D(\sB)$; since $E$ is compact in $\sB$, this formula
shows that $E[n]$ is compact in $D(\sB)$. Therefore we must show
that the homomorphisms
\[\bigoplus_{i\in I} \sB(E,H^n(RF C_i))\to \sB(E, H^n(RF \bigoplus_{i\in I} C_i)).\]
are bijective. By (ii), the spectral sequence
\[\sB(E,R^pFH^q(C))\Rightarrow \sB(E,H^{p+q}(RFC))\]
converges for any $C\in D(\sA)$. Thus it suffices to show that the
homomorphisms
\[\bigoplus_{i\in I} \sB(E,R^pFH^q(C_i))\to \sB(E,R^pFH^q(\bigoplus_{i\in I} C_i))\]
are bijective. By (i), this follows from the compactness of $E$.

c) Keep the notation of b). We may test \eqref{eqA.3} on the
$E_\alpha$'s. By their compactness, we must show that the
composition
$$
\bigoplus_{i\in I} D(\sB)(E_\alpha, RF(C_i)[q])\to D(\sB)(E_\alpha,\bigoplus_{i\in I} RF(C_i)[q])
\to D(\sB)(E_\alpha,RF(\bigoplus_{i\in I} C_i)[q])
$$
is an isomorphism for all $q\in\Z$. By adjunction, it is transformed
into
\[\bigoplus_{i\in I} D(\sA)(G(E_\alpha), C_i[q])\to D(\sA)(G(E_\alpha),\bigoplus_{i\in I} C_i[q])\]
which is an isomorphism since the $G(E_\alpha)$ are compact.
\end{proof}

Finally, we need a practical sufficient condition to ensure that, in
Condition (i) of Proposition \ref{pA.1} b), the case $p=0$ implies
the cases $p>0$. This is given by the classical

\begin{lemma}\label{lA.1} Suppose that $F$ is strongly additive and that, in $\sA$, infinite direct sums of injectives are $F$-acyclic. Then $R^pF$ is strongly additive for any $p>0$.
\end{lemma}

\begin{proof} D\'ecalage.
\end{proof}

\end{document}